\documentclass [12pt]{article}
\usepackage{amsfonts}
\usepackage{amsthm}
\usepackage{CJK}
\usepackage{amssymb}
\usepackage{amsmath}    % need for
\usepackage{graphicx}
\usepackage{floatrow}   % need for figures
\usepackage{verbatim}   % useful for program listings
\usepackage{color}      % use if color is used in text
\usepackage{subfigure}
\usepackage{epstopdf} % use for side-by-side figures
\usepackage{epsfig}
\usepackage{subfigure}
\usepackage{float}
\usepackage{cases}
\usepackage{threeparttable}
\usepackage{caption}
\usepackage{bm}
\newtheorem{theorem}{Theorem}[section]

\newtheorem{prop}[theorem]{Proposition}

\theoremstyle{definition}
\newtheorem{remark}{Remark}[section]
\newtheorem{example}[remark]{Example}
\numberwithin{equation}{section}
\newcommand{\norm}[1]{\left\Vert#1\right\Vert}
\newcommand{\abs}[1]{\left\vert#1\right\vert}

\newcommand{\average}[1]{\ensuremath{\{\!\!\{#1\}\!\!\}} }
\newcommand{\jump}[1]{\ensuremath{[\![#1]\!]} }

\begin{document}

\title{Discontinuous Galerkin methods for the Ostrovsky-Vakhnenko equation}
\date{}
\author{Qian Zhang\thanks{School of Mathematical Sciences, University of Science and Technology of China, Hefei, Anhui 230026, P.R. China.  E-mail: gelee@mail.ustc.edu.cn.}
\and Yinhua Xia\thanks{Corresponding author. School of Mathematical Sciences, University of Science and Technology of China, Hefei, Anhui 230026, P.R. China.  E-mail: yhxia@ustc.edu.cn. Research supported by NSFC grant  11871449, and a grant from Laboratory of Computational Physics (No. 6142A0502020817).}
}

\maketitle

\begin{abstract}
In this paper, we develop discontinuous Galerkin (DG) methods for the Ostrovsky-Vakhnenko (OV) equation, which yields the shock solutions and singular soliton solutions, such as peakon, cuspon and loop solitons.  The OV equation has also been shown to have a bi-Hamiltonian structure. We directly develop the energy stable or Hamiltonian conservative discontinuous Galerkin (DG) schemes for the OV equation. Error estimates for the two energy stable schemes are  also  proved.  For some singular solutions, including cuspon and loop soliton solutions, the hodograph transformation is adopted to transform the OV equation or the generalized OV system to the coupled dispersionless (CD) system.   Subsequently, two DG schemes are constructed for the transformed CD system. Numerical experiments are provided to demonstrate the accuracy and capability of the DG schemes, including shock solution and, peakon, cuspon and loop soliton solutions.
\end{abstract}
\vskip 4 cm


\section{Introduction}
In this paper, we consider the initial value problem of the Ostrovsky-Vakhnenko equation
\begin{equation}\label{Eqn:OV}
\begin{cases}
(u_t + uu_x)_x + \gamma u = 0, \ x \in I = [a,b],\ t>0, \\
u(x,0) = u_0(x).
\end{cases}
\end{equation}
which can be viewed as a particular limit of the generalized Korteweg-de Vries (KdV) equation
\begin{equation}\label{Eqn:GKDV}
(u_t + uu_x + \beta u_{xxx})_x + \gamma u  = 0.
\end{equation}
This equation \eqref{Eqn:GKDV} was derived in \cite{Hunter1990_LAM} as a model to describe the small-amplitude long waves on a shallow rotating fluid. Concerning the structure of this equation, it has a purely dispersive term. Although it has the same nonlinear term of the KdV equation, the dispersive terms are different. When $\beta = 0 $, there is no high-frequency dispersion. In \cite{Vakhnenko_1992_JPA}, Vakhnenko uses \eqref{Eqn:OV} to describe high frequency waves in a relaxing medium. In a series of papers \cite{Vakhnenko_1992_JPA, Parkes_1993_JPA, Vakhnenko_1998_non}, its integrability was established by deriving explicit solutions. It is known under different names in some literatures, such as the reduced Ostrovsky equation, the Ostrovsky-Hunter equation, the short-wave equation and the Vakhnenko equation, in this paper, we call \eqref{Eqn:OV} as Ostrovsky-Vakhnenko (OV) equation. The Ostrovsky-Vakhnenko equation has two properties that appear to be generic,
\begin{itemize}
\item Travelling waves that exist only up to a maximum limiting amplitude,
\item Limiting waves that have corners, i.e., a slope discontinuity.
\end{itemize}

The OV equation has peakon, shock and wave breaking phenomena even for smooth initial conditions in finite time. In \cite{Grimshaw_2012_SAM, Liu_2010_JMA, Hunter1990_LAM}, the authors have discussed the condition for wave breaking. Some exact solutions including periodic solution, and solitary traveling wave solution are investigated in \cite{Hunter1990_LAM, Parkes_2007_CSF, Parkes_2008_CSF}. Well-posedness results can be found in \cite{Gui_2007_CPDE, Linares_2006_JDE, Varlamov_2004_DCDs}. Through the hodograph transformation, \cite{Feng_2015_JPMT, Feng_2017_JPMT} provide the  cuspon and loop soliton solutions for the generalized OV system. Several numerical methods are proposed for the OV equation such as Fourier pseudo-spectral methods \cite{Grimshaw_2012_SAM} and a finite difference scheme based on the Engquist-Osher scheme \cite{2017_Coclite_BIT, 2018_Ridder_BIT}. Additionally, rigorous numerical analysis of the OV equation is concluded by Coclite, Ridder and Risebro  in  \cite{2017_Coclite_BIT, 2018_Ridder_BIT}, including the convergence results and  the existence of entropy solution. In \cite{Brunelli_2013_CNSNS}, bi-Hamiltonian structure of the OV equation is confirmed, i.e., the OV equation has infinite conservative quantities, in which we investigate energy $E$ and Hamiltonian $H$ as
\begin{align}
 E = \int_I u^2 dx, \mbox{  } H = \int_I -\frac{1}{6}u^3 + \frac{1}{2}(\partial^{-1}u)^2 dx.
\end{align}
The development of our numerical schemes are based on these two conservative quantities.  As the conservative methods for KdV equation \cite{Bona2013_MC, Xing2016_CiCP, Zhang2019CiCP}, Zakharov system \cite{Xia2010JCP}, Schr\"odinger-KdV system \cite{Xia2014CiCP}, short pulse equation \cite{Zhang2019_arxiv}, etc., various conservative numerical schemes are proposed to ``preserve structure". Usually, the conservative schemes can help reduce the phase error along the long time evolution.

%This OV equation describes the short wave perturbation in a relaxing medium.

The DG method was first introduced in 1973 by Reed and Hill in \cite{Reed1973} for solving steady state linear hyperbolic equations. The important ingredient of this method is the design of suitable inter-element boundary treatments (so called numerical fluxes) to obtain highly accurate and stable schemes in several situations. Within the DG framework, the method was extended to deal with derivatives of order higher than one, i.e., local discontinuous Galerkin (LDG) method. The first LDG method was introduced by Cockburn and Shu in \cite{Shu1998_Siam} for solving convection-diffusion equation. Their work was motivated by the successful numerical experiments of Bassi and Rebay \cite{Bassi1997_JCP} for compressible Navier-Stokes equations. Later, Yan and Shu developed an LDG method for a general KdV type equation containing third order
derivatives in \cite{Yan2002_Siam}, and they generalized the LDG method to PDEs with fourth and fifth spatial derivatives in \cite{Yan2002_JSC}. Levy, Shu and Yan \cite{Levy2004_JCP} developed LDG methods for nonlinear dispersive equations that have compactly supported traveling wave solutions, the so-called compactons. More recently, Xu and Shu further generalized the LDG method to solve a series of nonlinear wave equations \cite{Xu2004_JCM,Xu2005_JCP,Xu2005_PDNP,Xu2006_CMAME, Zhang2019_Jsc}. We refer to the review paper \cite{Xu2010_CiCP} of LDG methods for high-order time-dependent partial differential equations.

Here, we adopt the DG method as a spatial discretization to construct high order accurate numerical schemes for the OV equation.
 For general solutions,  the Hamiltonian conservative DG scheme and the energy stable schemes that contain the DG scheme and the integration DG scheme are developed. The energy stable schemes work for the smooth, peakon and shock solutions. The Hamiltonian conservative DG scheme can handle the smooth, peakon solutions and preserve the Hamiltonian spatially. The stability and conservation refer to the semi-discrete properties. For the time discretization, we use the so called total variation diminishing (TVD) or strong stability preserving (SSP) Runge-Kutta methods in \cite{1988_Shu_JCP, Gottlieb_2001_SIAM}. For some singular soliton solutions, we utilize the hodograph transformation to transform the OV equation to a coupled dispersionless (CD) type system, and then develop the DG scheme for the transformed CD system.

The paper is organized as follows. In Section 2, we directly construct two energy stable and Hamiltonian conservative DG schemes for the OV equation. We provide proofs of $L^2$ stability and Hamiltonian conservation. Suboptimal error estimates of the two energy stable schemes are also proved in this section.   For some singular soliton solutions, including loop and cuspon solitons, we transform the OV equation to the CD system via the hodograph transformation in Section 3. Subsequently, two DG schemes are developed for the CD system to obtain the numerical solutions for the OV equation indirectly. Some numerical experiments are presented in Section 4 to show the results of approximation. This paper is concluded in Section 5.

\section{The DG methods for the Ostrovsky-Vakhnenko equation}\label{sec2}
In this section, we develop two kinds of DG methods for the OV equation \eqref{Eqn:OV}, including energy stable schemes and the Hamiltonian conservative DG scheme.

\subsection{Notations}\label{notation}

We denote the mesh $\mathcal{T}_h $ by $I_j = [x_{j-\frac{1}{2}}, x_{j+\frac{1}{2}}] $ for $j = 1,\ldots, N$, with the cell center denoted by $ x_j = \frac{1}{2}(x_{j-\frac{1}{2}}+x_{j+\frac{1}{2}})$. The cell size is $\Delta x_j = x_{j+\frac{1}{2}}- x_{j-\frac{1}{2}} $ and $ h = \max\limits_{1\leq j\leq N} \ \Delta x_j$.  The finite element space as the solution and test function space consists of piecewise polynomials
$$V_h^k = \{v:v|_{I_j} \in P^k(I_j); 1\leq j\leq N\},$$
 where $P^k(I_j)$ denotes the set of polynomial of degree up to $k$ defined on the cell $I_j$. Notably, the functions in $V_h^k$ are allowed to be discontinuous across cell interfaces. The values of $u$ at $x_{j+\frac{1}{2}}$ are denoted by $u_{j+\frac{1}{2}}^-$ and $u_{j+\frac{1}{2}}^+$, from the left cell $I_j$ and the right cell $I_{j+1}$, respectively. Additionally, the jump of $u$ is defined as $\jump{u} = u^+ - u^- $, the average of $u$ as $\average{u} = \frac{1}{2}(u^+ + u^-)$.

After the hodograph transformation, the spatial variable change into $y$ from $x$. We denote the mesh $\mathcal{T}'_h $ by $I'_j = [y_{j-\frac{1}{2}}, y_{j+\frac{1}{2}}] $ for $j = 1,\ldots, N$. As the same definition on variable $x$,  we have $y_j, \Delta y_j,h' = \max\limits_{1\leq j\leq N} \ \Delta y_j$.

To simplify expressions, we adopt the round bracket and angle bracket for the $L^2$ inner product on cell $I_j$ and its boundary
\begin{equation}\label{bracket_def}
\begin{split}
(u, v)_{I_j}& = \int_{I_j} uv dy,\\
<\hat{u}, v>_{I_j}& = \hat{u}_{j+\frac{1}{2}}v_{j+\frac{1}{2}}^- - \hat{u}_{j-\frac{1}{2}}v_{j-\frac{1}{2}}^+
\end{split}
\end{equation}
for one dimensional case.

\subsection{The Energy stable schemes}\label{sec:energy DG}
In this section, we develop two DG schemes with $L^2$ energy stability, and for smooth solutions, suboptimal order of accuracy $(k+\frac{1}{2})$-$th$ is proved for these two DG schemes. To distinguish other DG schemes in this paper, we call it the energy stable DG scheme and energy stable integration DG scheme for the OV equation \eqref{Eqn:OV}.

\subsubsection{The DG scheme for the OV equation}
First, we divide the OV equation into a first order system
\begin{align}\label{eqn:split_OV}
\begin{cases}
u_t + (\frac{1}{2}u^2)_x  + \gamma v = 0, \ x\in[a,b],\\
v_x = u.
\end{cases}
\end{align}

An extra constraint for $v$ is necessary to ensure the unique solution of the initial value problem \eqref{Eqn:OV}. Referring to \cite{2017_Coclite_BIT, 2018_Ridder_BIT}, there are two cases of constraints for $v$:
\begin{itemize}
\item  For the  Dirichlet boundary problem, the fixed boundary condition for $v$ is adopted,
\begin{align}\label{eqn: boundary}
  %v(a,t)= \int_a^x u \ d\xi \ \text{or} \ v(b,t) =  \int_x^b u \ d\xi.
  v(a,t)= 0 \ \ \text{or} \ \ v(b,t) = 0.
\end{align}
\item  For the periodic boundary problem, the zero mean condition  $\int_I v dx = 0$ is adopted.
\end{itemize}
%For the constraint of numerical solution, we will give a more specific explanation in
$\mathbf{Scheme \ 1}$  :
The  energy stable DG scheme is formulated as follows: Find the numerical solutions $u_h, v_h \in V_h^k$,  for all test functions $\phi, \varphi \in V_h^k$, such that
\begin{subnumcases}{\label{scheme:energy_con}}
((u_h)_t, \phi)_{I_j} + <\widehat{f(u_h)}, \phi>_{I_j} - (f(u_h), \phi_x)_{I_j}  + \gamma(v_h, \phi)_{I_j} = 0,\label{scheme:energy_con1} \\
<\widehat{v_h}, \varphi>_{I_j} - (v_h, \varphi_x)_{I_j} = (u_h,\varphi)_{I_j}\label{scheme:energy_con2}
\end{subnumcases}
where $f(u) = \frac{1}{2}u^2 $.  The ``hat" terms in \eqref{scheme:energy_con} are the so-called ``numerical fluxes'', which are functions defined on the cell boundary from integration by parts and should be designed based on different guiding principles for different PDEs to ensure the stability and local solvability of the intermediate variables. To introduce some dissipation of $L^2$ energy, we adopt the dissipative numerical fluxes as
\begin{align}
&\widehat{v_h} = \begin{cases} v_h^-, \; & \gamma > 0 , \\
                               v_h^+, \; &  \gamma < 0,
                  \end{cases}, \label{eqn: flux1}\\
&\widehat{f(u)} = \frac{1}{2}(f(u^+)+f(u^-)-\alpha(u^+-u^-)), \ \alpha = \max\limits_u\abs{f'(u)}.\label{eqn: flux2}
\end{align}

The flux $\widehat{f(u)}$ we consider here is the Lax-Friedrichs flux,  which is regarded as a dissipative flux. The numerical flux $\widehat{v_h}$ depends on the sign of the parameter $\gamma$. When $\gamma$ is positive, $\widehat{v_h}$ is taken as $v_h^-$, otherwise, $\widehat{v_h} = v_h^+$. Numerically, the DG scheme with numerical fluxes \eqref{eqn: flux1}, \eqref{eqn: flux2} can achieve $(k+1)$-$th$ order of accuracy.
 %The flux here is not unique, for example, we could also choose conservative flux
%\begin{align}\label{conservative_flux}
%\widehat{f(u)} =  \begin{cases} \frac{\jump{F(u)}}{\jump{u}}, \; &\jump{u} \neq 0, \\ f(\average{u}),\; &\jump{u} = 0.\end{cases},
%\end{align}

%However, the DG scheme with conservative flux \eqref{conservative_flux} is only $k$-$th$ order for odd $k$, $(k+1)$-$th$ order for even $k$.

$\mathbf{Scheme \ 2}$  :
Alternatively, we can integrate the equation $(v_h)_x = u_h $ directly instead of the DG scheme \eqref{scheme:energy_con2}. Therefore, the energy stable integration DG scheme is defined as: Find the numerical solutions $u_h \in V_h^k$, $v_h \in V_h^{k+1}$, for all test functions $\phi \in V_h^k$, such that
\begin{subnumcases}{\label{scheme:energy_con_integration}}
((u_h)_t, \phi)_{I_j} + <\widehat{f(u_h)}, \phi>_{I_j} - (f(u_h), \phi_x)_{I_j}  + \gamma(v_h, \phi)_{I_j} = 0,\label{scheme:energy_con_integration1} \\
 v_h(x,t)\mid_{I_j}  = v_h(x_{j+\frac{1}{2}},t) - \int_x^{x_{j+\frac{1}{2}}}u_h(\xi,t) \ d\xi.\label{scheme:energy_con_integration2}
\end{subnumcases}
The equation \eqref{scheme:energy_con_integration2} can also be replaced by
\begin{align}\label{scheme:energy_con_integration3}
v_h(x,t)\mid_{I_j}  = v_h(x_{j-\frac{1}{2}},t) + \int^x_{x_{j-\frac{1}{2}}}u_h(\xi,t) \ d\xi,
\end{align}
which depends on the boundary condition of $v_h$. For the constraint of numerical solution $v_h$, we will give a more specific explanation in next section.
%Here and after, without loss of generality, we assume that $\gamma = 1$.

\subsubsection{Algorithm flowchart}\label{sec:algorithm}
In this part, we give some details related to the implementation of our numerical Scheme 1 and Scheme 2. We can see that the equation \eqref{scheme:energy_con1}, \eqref{scheme:energy_con_integration1} are exactly the same. The main difference between Scheme 1 and Scheme 2 lies in  \eqref{scheme:energy_con2} and \eqref{scheme:energy_con_integration2}, respectively,  which we will explain in Step 1.

{\bf{Step 1}} : First, we obtain $v_h$ from $u_h$ by \eqref{scheme:energy_con2} in Scheme 1, or \eqref{scheme:energy_con_integration2} in Scheme 2.
\begin{itemize}
\item In Scheme 1: From the equation \eqref{scheme:energy_con2}, we have the following matrix form,
    \begin{align}
     \mathbf{A}\mathbf{v}_h= \mathbf{u}_h.
    \end{align}
% If we take Dirichlet boundary for $u_h$, then the constraint of $v_h$ can be given by the value of a fixed endpoint, such as $v_h(b,t) = v(b,t)$.
Here, $\mathbf{u}_h, \mathbf{v}_h$ are the vectors containing the degrees of freedom for $u_h$ and $v_h$, respectively. The size of matrix $\mathbf{A}$ is $(N*(k+1))\times(N*(k+1)) $, $N $ is the number of spatial cells and $k$ is the degree of the approximate space $V_h^k$. However, if $v_h$ is periodic, the matrix $\mathbf{A}$ is under-determined and the rank of $\mathbf{A}$ is $N*(k+1)-1$. Therefore, as a replacement, the zero mean condition $\int_I v_h = 0$ helps determine a unique solution.
%Then the number of variables equals the number of equation so that we will solve a unique $v_h$.

\item In Scheme 2:
Under the fixed boundary condition of $v_h$, we choose $v_h(x_{N+\frac{1}{2}},t) = 0$   in \eqref{scheme:energy_con_integration2}  or $v_h(x_{\frac{1}{2}},t) = 0$ in \eqref{scheme:energy_con_integration3}. And then $v_h$ can be solved cell by cell.  For the zero mean condition, we also choose $v_h(x_{N+\frac{1}{2}},t) = 0 $. According to the equation  \eqref{scheme:energy_con_integration2}, we can get $v_h$ on each cell $I_j$. Subsequently, we check our zero mean condition by calculating the value of $\bar{v}_h = \int_I v_h $. Generally,  $\bar{v}_h$ will not be zero. Thereafter, a modification is done for the value of $v_h(x_{j+\frac{1}{2}},t)$,
$$ v_h(x_{j+\frac{1}{2}},t) = v_h(x_{j+\frac{1}{2}},t) - \frac{\bar{v}_h}{b-a},\ j = 1,\ldots, N. $$
Subsequently, we obtain a numerical solution $v_h$ that satisfies the condition $\int_I v_h = 0$.
\end{itemize}

{\bf{Step 2}} :  Substituting $v_h$ into the equation \eqref{scheme:energy_con1}, we have
  \begin{align}
  (\mathbf{u}_h)_t = \mathbf{res}(\mathbf{u_h}, \mathbf{v_h}).
  \end{align}
By choosing a suitable ODE solver, such as Runge-Kutta time discretization method, we will finally implement these two numerical schemes.

\subsubsection{$L^2$ stability of the energy stable schemes}
The $L^2$ stability of Scheme 1 and Scheme 2 are presented in Proposition \ref{prop1} and \ref{prop2}, respectively. This is the reason why we call Scheme 1 as the energy stable DG scheme, and Scheme 2 as the energy stable integration DG scheme.

\begin{prop}\label{prop1}($L^2$ stability for Scheme 1)

The semi-discrete DG scheme \eqref{scheme:energy_con} with fluxes \eqref{eqn: flux1}, \eqref{eqn: flux2} is an $L^2$ energy stable DG scheme, i.e.,
\begin{equation}
\frac{d}{dt} E(u_h) = \frac{d}{dt}\int_I u_h^2 dx \leq 0.
\end{equation}

\end{prop}

\begin{proof}
We take the test function $\phi = u_h$, $\varphi = \gamma v_h$ in scheme \eqref{scheme:energy_con}, thereafter, we obtain
\begin{align}
&((u_h)_t, u_h)_{I_j} + <\widehat{f(u_h)}, u_h>_{I_j} - (f(u_h), (u_h)_x)_{I_j}  + \gamma(v_h, u_h)_{I_j} = 0, \\
&<\widehat{v_h}, v_h>_{I_j} - (v_h, (v_h)_x)_{I_j} = (u_h, v_h)_{I_j}.
\end{align}
After applying summation of the above-mentioned two equations, we have
\begin{align}\label{eqn:cell entropy}
((u_h)_t, u_h)_{I_j}  + \Phi_{j+ \frac{1}{2}} - \Phi_{j-\frac{1}{2}} + \Theta_{j-\frac{1}{2}} = 0
\end{align}
where the numerical entropy flux is
\begin{equation}
\Phi = \gamma \widehat{v_h}(v_h^-)  -  \frac{\gamma}{2}(v_h^-)^2 + \widehat{f(u_h)}u_h^- - F(u_h^-)
\end{equation}
and the extra term $\Theta$ is given by
\begin{equation}\label{eqn:theta}
\begin{split}
\Theta = & -\gamma \widehat{v_h}\jump{v_h} - \gamma(\frac{1}{2}(v^-_h)^2 + \frac{1}{2}(v^+_h)^2) - \widehat{f(u_h)}\jump{u_h} + \jump{F(u_h)} \\
       = & \gamma(-\widehat{v_h} + \average{v_h})\jump{v_h} + (f(\xi) - \average{f(u_h)})\jump{u_h} + \frac{1}{2}\alpha\jump{u_h}^2.
\end{split}
\end{equation}
The choice of $\widehat{v_h}$ \eqref{eqn: flux2} can guarantee that the first term of \eqref{eqn:theta} is non-negative. According to the monotonicity of the numerical flux $f(\uparrow,\downarrow)$, we divide the above-mentioned equation into two cases:
\begin{align*}
 u^-\leq \xi \leq u^+,  \ (f(\xi) - \average{f(u)})\jump{u} \geq 0, \\
 u^+\leq \xi \leq u^-,  \ (f(\xi) - \average{f(u)})\jump{u} \geq 0.
\end{align*}
Thereafter,  we find that the whole term $\Theta$ is non-negative. Summing up the cell entropy equalities \eqref{eqn:cell entropy} with the periodic boundary condition or homogeneous Dirichlet boundary condition, we have the energy stability as
\begin{equation}
(u_h,(u_h)_t)_{I} \leq 0,
\end{equation}
i.e., $L^2$ energy stability of the DG scheme \eqref{scheme:energy_con} for the OV equation.
\end{proof}

\begin{prop}\label{prop2}($L^2$ stability for Scheme 2)

The semi-discrete DG scheme \eqref{scheme:energy_con_integration} is an $L^2$ energy stable scheme, i.e.,
\begin{equation}
\frac{d}{dt} E(u_h) = \frac{d}{dt}\int_I u_h^2 dx \leq 0.
\end{equation}
\end{prop}

\begin{proof}
We take test function $\phi = u_h$ in \eqref{scheme:energy_con_integration1},
\begin{align}\label{eqn:prop2_eqn}
((u_h)_t, u_h)_{I_j} + <\widehat{f(u_h)}, u_h>_{I_j} - (f(u_h), (u_h)_x)_{I_j}  + \gamma(v_h, u_h)_{I_j} = 0,
\end{align}
Additionally, following the idea of Proposition \ref{prop1}, for the nonlinear term $f(u)$, we can have a stable property. There is an extra term $\gamma(v_h, u_h)_{I_j}$ required to be estimated.
The Scheme 2, which is also called the integration DG method, is based on $(v_h)_x = u_h$, then
\begin{align}
\gamma(v_h, u_h)_{I_j} = \gamma(v_h, (v_h)_x)_{I_j} = \frac{\gamma}{2} ((v_h)^2_{j+\frac{1}{2}} - (v_h)^2_{j-\frac{1}{2}}).
\end{align}
Due to the continuity of $v_h$ and the periodic or homogeneous Dirichlet boundary condition,  we obtain the result of $L^2$ stability after summing up the equation \eqref{eqn:prop2_eqn} over all cells,
\begin{equation}
\frac{d}{dt} E(u_h) = \frac{d}{dt}\int_I u_h^2 dx \leq 0.
\end{equation}

\end{proof}

\subsubsection{Error estimates of the energy stable schemes}
In this section, the a-priori error estimate of  Scheme 1 \eqref{scheme:energy_con} and Scheme 2 \eqref{scheme:energy_con_integration} will be stated. Referring to the procedure in \cite{Zhang2018_SIAM, Xu2007_CMAME}, we will give the brief proofs in the subsequent descriptions. Without loss of generality, we let $\gamma = 1$ in this part.

First, we make some preparations for error estimate by giving necessary assumptions, projection and interpolation properties. The standard $L^2$ projection of a function $\zeta$ with $k+1$ continuous derivatives into space $V_h^k$, is denoted by $\mathcal{P}$, i.e., for each $I_j$
\begin{equation}\label{projection}
\begin{split}
&(\mathcal{P}\zeta - \zeta,\phi)_{I_j} =0, \ \forall \phi \in P^{k}(I_j),
\end{split}
\end{equation}
and the special projections $\mathcal{P}^{\pm}$ into $V_h^k$  satisfy, for each $I_j$
\begin{align}
(\mathcal{P}^{+}\zeta - \zeta,\phi)_{I_j} =0, \ \forall \phi \in P^{k-1}(I_j), \ \text{and} \ \mathcal{P}^{+}\zeta(y_{j-\frac{1}{2}}^+) = \zeta({y_{j-\frac{1}{2}}}),\label{eqn:special_projection1}\\
(\mathcal{P}^{-}\zeta - \zeta,\phi)_{I_j} =0, \ \forall \phi \in P^{k-1}(I_j),\ \text{and} \
\mathcal{P}^{-}\zeta(y_{j+\frac{1}{2}}^-) = \zeta({y_{j+\frac{1}{2}}}).\label{eqn:special_projection2}
\end{align}
For the projections mentioned above, it is easy to show \cite{1975_Ciarlet_NH} that
\begin{equation}\label{projection error}
\norm{ \zeta^e}_{L^2(I)} +  h^{\frac{1}{2}} \norm{ \zeta^e}_{L^{\infty}(I)} + h^{\frac{1}{2}}\norm{ \zeta^e}_{L^2({\partial I})}  \leq Ch^{k+1}
\end{equation}
where $\zeta^e =\zeta -\mathcal{P}\zeta  $ or $\zeta^e =\zeta -\mathcal{P}^{\pm}\zeta$, and the positive constant $C$ only depends on $\zeta$.

We will use an inverse inequality in the subsequent proofs. For $\forall u\in V_h^k$, there exists a positive constant $\sigma$ (we call it the inverse constant), such that
\begin{equation}\label{eqn:inverse inequality}
\norm{u}_{L^2(\partial{I})} \leq \sigma h^{-\frac{1}{2}}\norm{u}_{L^2(I)},
\end{equation}
where $\norm{u}_{L^2(\partial{I})}  = \sqrt{\sum\limits_{j=1}^{N}(u_{j+\frac{1}{2}}^-)^2 +(u_{j-\frac{1}{2}}^+)^2}$.

Additionally, to deal with the nonlinearity of the flux $f(u)$, we make a priori assumption that, there holds
\begin{equation}
\norm{u-u_h}_{L^2({I})}  \leq h
\end{equation}
for small enough $h$. Under this assumption, the error of $L^\infty$ norm satisfies
\begin{equation}\label{eqn:infty_estimate}
\norm{u-u_h}_{L^{\infty}({I})} \leq Ch^{\frac{1}{2}}, \ \norm{\zeta^e}_{L^{\infty}({I})} \leq Ch^{\frac{1}{2}}
\end{equation}
where $\zeta^e =\zeta -\mathcal{P}\zeta  $ or $\zeta^e =\zeta -\mathcal{P}^{\pm}\zeta$.

For the Scheme 1, we have below theorem to demonstrate the result of convergence for smooth exact solutions.
\begin{theorem}\label{thm1}
 It is assumed that the OV equation \eqref{eqn:split_OV} with periodic boundary condition has a sufficiently smooth exact solution $u$. The numerical solution $u_h$ satisfies the semi-discrete DG scheme \eqref{scheme:energy_con} with flux \eqref{eqn: flux1},\eqref{eqn: flux2}. For regular partitions of $I = (a, b)$, and the finite element
space $V^k_h$ with $k \geq 0$, there holds the following error estimate for small enough $h$
\begin{align}
\norm{ u - u_h}_{L^2(I)} \leq Ch^{k+\frac{1}{2}}.
\end{align}
\end{theorem}

\begin{proof}
First, we give the error equation between the exact solution and numerical solution,
\begin{align}
((u - u_h)_t,\phi)_{I_j} + & <f(u) - \widehat{f(u_h)},\phi>_{I_j} - (f(u) - f(u_h),\phi_x)_{I_j} + (v - v_h,\phi)_{I_j} \notag\\
+ & <v-\widehat{v_h},\varphi>_{I_j} - (v - v_h, \varphi_x)_{I_j} - (u - u_h, \varphi)_{I_j} = 0,
\end{align}
for all test functions $\phi, \varphi \in V_h^k$.
Thereafter, we define two bilinear forms
\begin{align}
&\mathcal{B}_j(u-u_h, v-v_h ; \phi,\varphi) \notag \\
&= ((u - u_h)_t,\phi)_{I_j}+ (v - v_h,\phi)_{I_j} + <v-\widehat{v_h},\varphi>_{I_j} - (v - v_h, \varphi_x)_{I_j} - (u - u_h, \varphi)_{I_j}
\end{align}
and
\begin{align}
&\mathcal{H}_j(f;u,u_h,\phi) = (f(u) - f(u_h),\phi_x)_{I_j} - <f(u) - \widehat{f(u_h)},\phi>_{I_j}.
\end{align}
After applying summation over all cells $I_j$, the error equation is expressed by
\begin{align}
\sum\limits_{j=1}^{N} \mathcal{B}_j(u-u_h, v-v_h ; \phi,\varphi) = \sum\limits_{j=1}^{N}\mathcal{H}_j(f;u,u_h,\phi).
\end{align}
Introducing notations
\begin{align}
&\xi^u = \mathcal{P}u-u_h, \ \eta^u = \mathcal{P}u - u, \\
&\xi^v = \mathcal{P}^-v-v_h,\ \eta^v = \mathcal{P}^-v - v,
\end{align}
and taking test functions $\phi = \xi^u, \varphi = \xi^v$, we have
\begin{align}
\sum\limits_{j=1}^{N} \mathcal{B}_j(\xi^u - \eta^u, \xi^v - \eta^v ; \xi^u,\xi^v) = \sum\limits_{j=1}^{N}\mathcal{H}_j(f;u,u_h,\xi^u).
\end{align}
For bilinear form $\mathcal{B}_j$, the following equation holds by projection properties \eqref{projection}-\eqref{eqn:special_projection2}
\begin{equation}\label{eqn:B_estimate}
%\sum\limits_{j=1}^{N} \mathcal{B}_j(\xi^u - \eta^u, \xi^v - \eta^v ; \xi^u,\xi^v) =
%(\xi^u_t, \xi^u)_{I} - (\eta^u_t, \xi^u)_{I} + \sum\limits_{j=1}^{N}[\xi^v]^2_{j+\frac{1}{2}} + (\eta^v, \xi^u)_{I}.
\sum\limits_{j=1}^{N} \mathcal{B}_j(\xi^u - \eta^u, \xi^v - \eta^v ; \xi^u,\xi^v) =
(\xi^u_t, \xi^u)_{I}  + \sum\limits_{j=1}^{N}[\xi^v]^2_{j+\frac{1}{2}} - (\eta^v, \xi^u)_{I}.
\end{equation}
For bilinear form $\mathcal{H}_j$, we follow the idea of \cite{Zhang2018_SIAM, Xu2007_CMAME} to present the estimate of $\mathcal{H}_j$,
\begin{align}\label{eqn:H_estimate}
&\sum\limits_{j=1}^{N}\mathcal{H}_j(f;u,u_h,\xi^u) \leq -\frac{1}{4}\alpha(\hat{f};u_h)\jump{\xi^u}^2 + (C + C_*h^{-1}\norm{u-u_h}^2_{L^{\infty}(I)}) h^{2k+1}\notag \\
&+ (C+C_*(\norm{\xi^u}_{L^{\infty}(I)} + h^{-1}\norm{u-u_h}^2_{L^{\infty}({I})}))\norm{\xi^u}^2_{L^{\infty}(I)}
\end{align}
where $\alpha(\hat{f};u_h)$ is non-negative, the constant $C_*$ is a positive constant depending on the maximum of $\abs{f''}$ or/and $\abs{f'''}$, the details are listed in \cite{Zhang2018_SIAM,Xu2007_CMAME}.

Combining the estimate equations \eqref{eqn:B_estimate} and \eqref{eqn:H_estimate}, we obtain the final error estimate as follows,
\begin{align}
&(\xi^u_t, \xi^u)_{I} + \frac{1}{4}\alpha(\hat{f};u_h)\jump{\xi^u}^2 + \sum\limits_{j=1}^{N}[\xi^v]^2_{j+\frac{1}{2}} \notag\\
&\leq (\eta^v, \xi^u)_{I} + (C + C_*h^{-1}\norm{u-u_h}^2_{L^{\infty}(I)}) h^{2k+1}\notag+(C+C_*(\norm{\xi^u}_{L^{\infty}(I)} + h^{-1}\norm{u-u_h}^2_{L^{\infty}({I})}))\norm{\xi^u}^2_{L^{\infty}(I)}.
\end{align}
Using Young's inequality and interpolation properties \eqref{eqn:infty_estimate}, we get
\begin{align*}
 \frac{1}{2}\frac{d}{dt}\norm{\xi^u}_{L^2(I)} \leq C\norm{\xi^u}_{L^2{(I)}} + Ch^{2k+1} + Ch^{2k+2}.
\end{align*}
Utilized Gronwall's inequality, the equation becomes
\begin{align*}
\norm{\xi^u}_{L^2(I)}^2 \leq  Ch^{2k+1}.
\end{align*}
Therefore, the result of Theorem \ref{thm1} is derived by triangle inequality and the interpolation inequality \eqref{projection error}.

\end{proof}

For the Scheme 2, we also state the following error estimate for smooth exact solutions.
\begin{theorem}\label{thm2}
It is assumed that the OV equation \eqref{eqn:split_OV} with the periodic boundary condition has a sufficiently smooth exact solution $u$. The numerical solution $u_h$ satisfies the integration DG scheme \eqref{scheme:energy_con_integration}. For regular partitions of $I = (a, b)$, and the finite element
space $V^k_h$ with $k \geq 0$, for adequately small $h$, there holds
\begin{align}\label{eqn:thm2}
\norm{ u - u_h}_{L^2(I)} \leq Ch^{k+\frac{1}{2}}.
\end{align}
\end{theorem}

\begin{proof}
Similarly, we give the error equations,
\begin{align}
&((u - u_h)_t,\phi)_{I_j} +  <f(u) - \widehat{f(u_h)},\phi>_{I_j} - (f(u) - f(u_h),\phi)_{I_j} + (v - v_h,\phi)_{I_j} = 0, \notag\\
& (v - v_h)_x = u - u_h \notag
\end{align}
for any test function $\phi \in V_h^k$. Thereafter, we define another bilinear form
\begin{align}
&\mathcal{\tilde{B}}_j(u-u_h, v-v_h ; \phi) = ((u - u_h)_t,\phi)_{I_j}+ (v - v_h,\phi)_{I_j}.
\end{align}
After applying summation over all cells $I_j$, the error equations are expressed by
\begin{align}
\sum\limits_{j=1}^{N} \mathcal{\tilde{B}}_j(u-u_h, v-v_h ; \phi) = \sum\limits_{j=1}^{N}\mathcal{H}_j(f;u,u_h,\phi).
\end{align}
We define notations
\begin{align}
&\xi^u = \mathcal{P}u-u_h, \ \eta^u = \mathcal{P}u - u, \\
&\xi^v = \tilde{v}-v_h,\ \eta^v = \tilde{v} - v
\end{align}
where $\tilde{v}_x = \mathcal{P^-}u, \tilde{v} \in V_h^{k+1}$, it satisfies $\xi^v_x = \xi^u$. By Poincar\'{e} inequality, the error of $\eta^v$
is controlled by $\eta^u$,
\begin{equation}
\norm{\eta^v}_{L^2(I_j)}\leq \norm{\eta^u}_{L^2(I_j)} \leq Ch^{k+1}.
\end{equation}

With the test function $\phi = \xi^u \in V_h^k $, we have
\begin{align}
\sum\limits_{j=1}^{N} \mathcal{\tilde{B}}_j(\xi^u - \eta^u, \xi^v - \eta^v ; \xi^u) = \sum\limits_{j=1}^{N}\mathcal{H}_j(f;u,u_h,\xi^u).
\end{align}
For bilinear form $\mathcal{\tilde{B}}_j$, the following equation holds by projection properties \eqref{projection}
\begin{align}\label{eqn:Bwan_estimate}
\sum\limits_{j=1}^{N} \mathcal{\tilde{B}}_j(\xi^u - \eta^u, \xi^v - \eta^v ; \xi^u,\xi^v) &=
(\xi^u_t, \xi^u)_{I} - (\eta^u_t, \xi^u)_{I} + (\xi^v, \xi^u)_{I} - (\eta^v, \xi^u)_{I} \notag\\
&= (\xi^u_t, \xi^u)_{I} + \sum\limits_{j=1}^{N}\jump{(\xi^v)^2}_{j+\frac{1}{2}} - (\eta^v, \xi^u)_{I} \\
& = (\xi^u_t, \xi^u)_{I}  - (\eta^v, \xi^u)_{I} \notag
\end{align}

Combining the estimate equations \eqref{eqn:Bwan_estimate} and \eqref{eqn:H_estimate}, we obtain the final error estimate,
\begin{align}
&(\xi^u_t, \xi^u)_{I} + \frac{1}{4}\alpha(\hat{f};u_h)\jump{\xi^u}^2 \notag\\
&\leq  (\eta^v, \xi^u)_{I} + (C+C_*(\norm{\xi^u}_{L^{\infty}(I)} + h^{-1}\norm{u-u_h}^2_{L^{\infty}({I})}))\norm{\xi^u}^2_{L^{\infty}(I)}\notag \\
&+ (C + C_*h^{-1}\norm{u-u_h}^2_{L^{\infty}(I)}) h^{2k+1}.
\end{align}
Similar to the procedure followed in deriving the proof of Theorem \ref{thm1}, we have \eqref{eqn:thm2} in Theorem \ref{thm2}.
%By Young's inequality and interpolation properties \eqref{eqn:infty_estimate}, we get
%\begin{align*}
% \frac{1}{2}\frac{d}{dt}\norm{\xi^u}_{L^2(I)} \leq C\norm{\xi^u}_{L^2{(I)}} + Ch^{2k+1} + Ch^{2k+2}
%\end{align*}
%Utilized Gronwall inequality, the equation becomes
%\begin{align*}
%\norm{\xi^u}_{L^2(I)}^2 \leq  Ch^{2k+1}.
%\end{align*}
%Thus the result of Theorem \ref{thm1} can be derived by triangle inequality and the interpolation inequality \eqref{projection error}.
\end{proof}

\subsection{The Hamiltonian conservative DG scheme}\label{sec:Hamiltonian DG}
In this section, we construct another DG scheme, which can preserve the Hamiltonian spatially. Therefore,  we call it the Hamiltonian conservative DG scheme for the OV equation.
%\subsubsection{The DG scheme for the OV equation}

We rewrite the OV equation as another first order system
\begin{align}\label{ODE:H}
\begin{cases}
u_t + w_x  + v = 0, \\
w = \frac{1}{2} u^2, \\
v_x = u. 
\end{cases}
\end{align}
The Hamiltonian conservative DG scheme is defined as: Find numerical solutions $u_h, v_h, w_h\in V_h^k$, for all test functions $\phi, \varphi, \psi \in V_h^k$, such that
\begin{subnumcases}{\label{scheme:H_con}}
((u_h)_t,\phi)_{I_j} + <\widehat{w_h}, \phi>_{I_j}  - (w_h,\phi_x)_{I_j}  + (v_h, \phi)_{I_j} = 0, \label{scheme:H_con1}\\
(w_h,\varphi)_{I_j} = (\frac{1}{2} u_h^2,\varphi)_{I_j},\label{scheme:H_con2} \\
 <\widehat{v_h}, \psi>_{I_j} - (v_h, \psi_x)_{I_j} = (u_h, \psi)_{I_j}.\label{scheme:H_con3}
\end{subnumcases}
The numerical fluxes are taken as
\begin{align}\label{eqn:flux_H}
\widehat{w_h} = \average{w_h}, \quad \widehat{v_h} = \average{v_h}.
\end{align}

The difference between the Hamiltonian conservative DG scheme and energy stable DG scheme in section \ref{sec:energy DG} is an $L^2$ projection \eqref{scheme:H_con2}. Therefore, the implementation of algorithm is similar to Section \ref{sec:algorithm}, the flowchart is omitted here.

\begin{remark}
We can still deal   $v_x = u$ by directly integrating it, similar to the method followed in equation \eqref{scheme:energy_con_integration2} or \eqref{scheme:energy_con_integration3}. To avoid unnecessary duplication, we do not repeat the process.
\end{remark}

%\subsubsection{Hamiltonian conservation}
\begin{prop}
The semi-discrete DG scheme \eqref{scheme:H_con} with fluxes \eqref{eqn:flux_H} is a Hamiltonian conservative DG scheme that can preserve the Hamiltonian spatially
\begin{equation}
\frac{d}{dt} H(u_h,v_h) = \frac{d}{dt} \int_I -\frac{1}{6}u_h^3 + v_h^2 dx = 0.
\end{equation}
\end{prop}

\begin{proof}
First, we take the time derivative of the equation \eqref{scheme:H_con3} to obtain
\begin{equation}\label{scheme:H_con4}
 <\widehat{(v_h)_t}, \eta>_{I_j} - ((v_h)_t, \eta_x)_{I_j} = ((u_h)_t, \eta)_{I_j}.
\end{equation}
As \eqref{scheme:H_con1},\eqref{scheme:H_con2},\eqref{scheme:H_con4} hold true for any test function in space $V_h^k$, we choose
\begin{equation}
\phi = (v_h)_t, \ \varphi = (u_h)_t, \  \eta = w_h.
\end{equation}
Using the selected fluxes and summing up the three above-mentioned equations \eqref{scheme:H_con1},\eqref{scheme:H_con2},\eqref{scheme:H_con4}, we have
\begin{align}\label{eqn:summation}
 ((u_h)_t,(v_h)_t)_{I_j} +& <\widehat{w_h}, (v_h)_t>_{I_j}  - (w_h,(v_h)_{tx})_{I_j}  + (v_h, (v_h)_t)_{I_j} \\
  +& <\widehat{(v_h)_t}, w_h>_{I_j} - ((v_h)_t, (w_h)_x)_{I_j}  - (\frac{1}{2} u_h^2,(u_h)_t)_{I_j} = 0. \notag
 \end{align}
 To eliminate the extra term $((u_h)_t,(v_h)_t)_{I_j}$, we take the test function $\eta = (v_h)_t$ in equation \eqref{scheme:H_con4}, and obtain
 \begin{equation}\label{scheme:H_con5}
 <\widehat{(v_h)_t}, (v_h)_t>_{I_j} - ((v_h)_t, (v_h)_{tx})_{I_j} = ((u_h)_t, (v_h)_t)_{I_j}.
\end{equation}
Substituting equation \eqref{scheme:H_con5} into \eqref{eqn:summation}, we finally get the following summation
\begin{align}
<\widehat{(v_h)_t}, (v_h)_t>_{I_j} - ((v_h)_t, (v_h)_{tx})_{I_j} + & <\widehat{w_h}, (v_h)_t>_{I_j}  - (w_h,(v_h)_{tx})_{I_j}  + (v_h, (v_h)_t)_{I_j} \\
  +& <\widehat{(v_h)_t}, w_h>_{I_j} - ((v_h)_t, (w_h)_x)_{I_j}  - (\frac{1}{2} u_h^2,(u_h)_t)_{I_j} = 0. \notag
\end{align}
We rewrite the above-mentioned equation into its equivalence form
\begin{align}\label{eqn:cell entropy_H}
(v_h, (v_h)_t)_{I_j}- (\frac{1}{2} u_h^2,(u_h)_t)_{I_j} + \Phi_{j+\frac{1}{2}}-\Phi_{j-\frac{1}{2}} + \Theta_{j-\frac{1}{2}} = 0
\end{align}
where the numerical entropy flux is given by
\begin{equation}
\Phi = \average{w_h}(v_h^-)_t  - \widehat{(v_h)_t}w_h^- - (v_h^-)_t w_h^- +\average{v_h}(v_h^-)  - \frac{1}{2}(v_h^-)^2
\end{equation}
and the extra term $\Theta$ is
\begin{equation}
\begin{split}
\Theta =  -\average{w_h}\jump{(v_h)_t} - \average{(v_h)_t}\jump{w_h}-& (w_h(v_h)_t)^- + (w_h(v_h)_t)^+ \\
+&\average{v_h}\jump{v_h}- \frac{1}{2}(v^-_h)^2 + \frac{1}{2}(v^+_h)^2  =  0.
\end{split}
\end{equation}
Summing up the cell entropy equalities \eqref{eqn:cell entropy_H} with periodic or homogeneous Dirichlet boundary condition, the Hamiltonian conservation is proved
\begin{equation}
(-\frac{1}{2}(u_h)^2,(u_h)_t)_{I} + (v_h,(v_h)_t)_{I} = 0,
\end{equation}
i.e., Hamiltonian conservative DG scheme for the OV equation.
\end{proof}

\section{The DG methods via the hodograph transformation}
In this section, we solve the singular solutions of the OV equation \eqref{Eqn:OV} by transforming it into a new coupled dispersionless type equation (CD system). This type of a method that solves numerical solutions by hodograph transformations is also applied in \cite{Zhang2019_arxiv} for the short pulse equation. Similar to the method followed in \cite{Zhang2019_arxiv}, a DG scheme and an integration DG scheme are constructed for the CD system. After obtaining the numerical solutions of the CD system, the profiles of solutions for the OV equation are obtained.

Through the hodograph transformation
\begin{equation}\label{eqn:Hodograph}
dx = \frac{1}{\rho} dy + u ds,\ dt =ds,
\end{equation}
we link the OV equation \eqref{Eqn:OV} with a new type CD system
\begin{align}\label{eqn:New_CD}
\begin{cases}
(\rho^{-1})_s = u_y, \\
 \rho u_{ys} + \gamma u = 0.
\end{cases}
\end{align}
Additionally, the same hodograph transformation can be applied to the two component OV system
\begin{align}\label{eqn:2Ostrovsky}
\begin{cases}
(u_{t} +uu_x)_x + \gamma u = c(1-\rho),\\
\rho_t + (\rho u)_x = 0. 
\end{cases}
\end{align}
What we need to solve here is the following CD system
\begin{align}
\begin{cases}
(\rho^{-1})_s = u_y, \\
 \rho u_{ys} + \gamma u+c(\rho-1) = 0.
\end{cases}
\end{align}
When $c = 0$, the two component OV system will degenerate to the OV equation.
For the sake of expression, we make $q = 1/\rho$, then we have
\begin{equation}\label{eqn:solve_reduced_O}
\begin{cases}
q_s = u_y,\\
u_{ys} + \gamma qu + c(1-q) = 0.
\end{cases}
\end{equation}
%The new type CD system can be written as
%\begin{align}\label{eqn:New_CD}
%&(\rho^{-1})_s = u_y, \\
%& \rho u_{ys} + \gamma u = 0.
%\end{align}
%whose corresponding hodograph transformation is
%\begin{equation}\label{eqn:Hodograph}
%dx = \frac{1}{\rho} dy + u ds,\ dt =ds
%\end{equation}
%For the sake of expression, we make $q = 1/\rho$, then have
%\begin{equation}\label{eqn:solve_reduced_O}
%\begin{cases}
%&q_s = u_y,\\
%& u_{ys} + \gamma qu = 0.
%\end{cases}
%\end{equation}
\subsection{The DG schemes for the CD system }
In this section, two DG schemes are constructed for the CD system \eqref{eqn:solve_reduced_O}, including the DG scheme and the integration DG scheme, the specific forms of which will be provided in Scheme 1 and Scheme 2, respectively.

We rewrite \eqref{eqn:solve_reduced_O} as a first order system
\begin{equation}\label{eqn:ODE_CD}
\begin{cases}
q_s = \omega , \\
\omega_s = -\gamma qu - c(1-q) , \\
\omega = u_y.
\end{cases}
\end{equation}

$\mathbf{Scheme \ 3}$ : The DG scheme for the CD system \eqref{eqn:solve_reduced_O} is defined as follows: Find $q_h, u_h,\omega_h \in V^k_h$, such that,
\begin{subnumcases}{\label{eqn:DG_reduce_O}}
((q_h)_s, \phi)_{I'_j}  = (\omega_h,\phi)_{I'_j}, \label{eqn:DG_reduce_O_1}\\
((\omega_h)_s,\varphi)_{I'_j} = -(\gamma q_h u_h, \varphi)_{I'_j} - (c(1-q_h),\varphi)_{I'_j} , \label{eqn:DG_reduce_O_2}\\
 (\omega_h,\psi)_{I'_j} = <\widehat{u_h}, \psi>_{I'_j} - (u_h, \psi_y)_{I'_j}, \label{eqn:DG_reduce_O_3}
\end{subnumcases}
for all test functions $\phi,\varphi,\psi \in V_h^k$. Here, the numerical flux is $\widehat{u_h} = u_h^+$.
After solving the numerical solution of the CD system, we can finally profile the singular solutions of the OV equation.

$\mathbf{Scheme \ 4}$:
Under the same DG framework \eqref{eqn:ODE_CD}, we use integration scheme deal with the equation $ u_y = \omega$. Here, we construct the integration DG scheme for the CD system \eqref{eqn:solve_reduced_O}: Find $q_h,\omega_h \in V_h^k$, $u_h \in V^{k+1}$, such that
\begin{subnumcases}{\label{eqn:DG_reduce_O_integration}}
((q_h)_s, \phi)_{I'_j}  = (\omega_h,\phi)_{I'_j}, \label{eqn:DG_reduce_O_integration1}\\
((\omega_h)_s,\varphi)_{I'_j} =  -(\gamma q_h u_h, \varphi)_{I'_j} - (c(1-q_h),\varphi)_{I'_j}, \label{eqn:DG_reduce_O_integration2}\\
u_h(y,s)\mid_{I_j}  = u_h(y_{j+\frac{1}{2}},s) - \int_y^{y_{j+\frac{1}{2}}}\omega_h(\xi,s) \ d\xi, \label{eqn:DG_reduce_O_integration3}
\end{subnumcases}
for all test functions $\phi,\varphi \in V_h^k$. Here, the boundary condition is taken as $u_h(y_{N+\frac{1}{2}},s) = u(b,s)$. The primary difference between this integration DG scheme \eqref{eqn:DG_reduce_O_integration}and the DG scheme  \eqref{eqn:DG_reduce_O} is the finite element space that the numerical solution $u_h$ belongs to. In this case, not only $u_h$ is in $V_h^{k+1}$ space, but $u_h$ is continuous. Numerically, this integration DG scheme can achieve $(k+2)$-$th$ order of accuracy for $u_h$, and $(k+1)$-$th$ order for $q_h, \omega_h$.
%\begin{remark}
%The two component OV system
%\begin{align}\label{eqn:2Ostrovsky}
%&(u_{t} +uu_x)_x + \gamma u = c(1-\rho),\\
%&\rho_t + (\rho u)_x = 0.\label{eqn:2Ostrovsky1}
%\end{align}
%can be transformed by the same hodograph transformation \eqref{eqn:Hodograph} to the following CD system
%\begin{align*}
%&(\rho^{-1})_s = u_y, \\
%& \rho u_{ys} + \gamma u+c(\rho-1) = 0.
%\end{align*}
%We can imitate the DG scheme \eqref{eqn:DG_reduce_O} to solve this two component OV system. Because of the similarities of the structure, the detail description of numerical schemes are omitted. We will illustrate this OV system in numerical experiments.
%\end{remark}

\subsection{Algorithm flowchart}
In this section, the processes of Scheme 3 and Scheme 4 are listed as follows:

 {\bf{Step 1}} : From the equations \eqref{eqn:DG_reduce_O_1},\eqref{eqn:DG_reduce_O_2}, we have
\begin{align*}
(\mathbf{q}_h)_s = \mathbf{Res}(\bm{\omega}_h),\\
(\bm{\omega}_h)_s = \mathbf{Res}(\mathbf{u}_h, \mathbf{q}_h).
\end{align*}
 The vectors $\mathbf{u}_h, \mathbf{q}_h,\bm{\omega}_h$ denote the freedoms of numerical solutions $u_h, q_h,\omega_h$.   TVD/SSP Runge-Kutta method is used for solving ${\omega}_h, {q}_h$.

{\bf{Step 2}} :
From \eqref{eqn:DG_reduce_O_3} or \eqref{eqn:DG_reduce_O_integration3}, the coefficients of $u_h$ can be solved from $
\omega_h$. The specific procedures we have illustrated in Section \ref{sec:algorithm}, we do not list further details here.
%\begin{itemize}
%\item Scheme 1:
%From the matrix form of the equation \eqref{eqn:DG_reduce_O_3}, the coefficients of $u^{n+1}_h$  can be solved by
%    \begin{align}
%    \mathbf{A}\mathbf{u}_h = \bm{\omega}_h
%    \end{align}
%\item Scheme 2 : From the equation \eqref{eqn:DG_reduce_O_integration3}, $u^{n+1}_h$ can be solved on each cell $I_j$.
%\end{itemize}
%\section{The DG methods for generalized KdV equation}

\section{Numerical Experiments}
In this section, some numerical experiments are presented to show the convergence rate and capability of our numerical schemes. The time discretization method is the TVD/SSP Runge-Kutta method \cite{1988_Shu_JCP, Gottlieb_2001_SIAM}. We take the time step as $\Delta t = 0.1\Delta x $ with our uniform spatial meshes for all experiments. Different solutions of the OV equation are calculated in this part, including not only smooth, shock solution, but peakon, cuspon and loop soliton solutions.

\begin{example}\label{ex:smooth}$\bf{ Smooth\ solution}$

In this example, a smooth solution is used to test the accuracy and convergence rate of our numerical schemes with periodic boundary condition. The initial condition is taken as
\begin{equation}
u_0(x) = \sin(x), \ x\in [0,2\pi].
\end{equation}
We fix the exact solution as
\begin{align}\label{solution:smooth}
u(x,t) = \sin(x+t),
\end{align}
we add a source term $f = \cos{2(x+t)}$ to make sure the equation holds, that is,
\begin{equation}
(u_t + (\frac{1}{2}u^2)_x)_x  + u = f. \\
\end{equation}
We record the errors, orders of accuracy at time $T = 1$ for two DG schemes in Table \ref{tab: acc_test}. For the energy stable DG scheme, the convergence rate of $L^2$ and $L^{\infty}$ error is $(k+1)$-$th$ order for the variable $u$.  For the Hamiltonian conservative DG scheme, there is $k$-$th$ order of accuracy for odd $k$, and $(k+1)$-$th$ order for even $k$. In Table \ref{tab: acc_testv}, we compare the energy stable DG scheme and integration DG
 scheme on the variable $v$, the integration DG scheme is one order higher than the DG scheme on  the variable $v$. However, the final numerical solution $u$ belongs to space $V_h^k$, therefore, the convergence rate for variable $u$ is still $(k+1)$-$th$ order rather than $(k+2)$-$th$. In the subsequent examples, we do not emphasize the differences between the two aforementioned schemes.

\begin{table}[!htp]
\begin{center}
\begin{tabular}{|c|c|cccc|cccc|}
  \hline

          &   N       &$ \norm{u-u_h}_{L^2}$    & order   &$ \norm{u-u_h}_{L^\infty}$   & order
                     &$ \norm{u-u_h}_{L^2}$    & order   &$ \norm{u-u_h}_{L^\infty}$   & order\\\hline
           &           &  \multicolumn{4}{c|}{The energy stable DG scheme} & \multicolumn{4}{c|}{The Hamiltonian conservative DG scheme}\\\hline

 $P^1$  &  20  &1.91E-03 &-- &1.33E-02 &--
                 &2.17E-02 &-- &2.45E-01 &--\\

          &  40  &4.74E-04 &2.01 &3.35E-03 &1.99
                 &1.32E-02 &0.72 &1.44E-01 &0.77\\

&  80  &1.18E-04 &2.00 &8.38E-04 &2.00
&7.48E-03 &0.82 &1.32E-01 &0.12\\

& 160  &2.95E-05 &2.00 &2.10E-04 &2.00
&3.94E-03 &0.92 &8.40E-02 &0.65\\

& 320  &7.37E-06 &2.00 &5.24E-05 &2.00
&2.00E-03 &0.98 &4.87E-02 &0.79\\\hline

$P^2$&  20  &1.02E-04 &--  &8.00E-04 &--
            &1.58E-04 &-- &1.10E-03 &--\\

&  40  &1.44E-05 &2.82 &1.14E-04 &2.81
&1.33E-05 &3.57 &1.06E-04 &3.39\\

&  80  &1.93E-06 &2.90 &1.90E-05 &2.58
&1.21E-06 &3.46 &1.04E-05 &3.34\\

& 160  &2.56E-07 &2.91 &3.00E-06 &2.66
&1.16E-07 &3.38 &1.08E-06 &3.27\\

& 320  &3.53E-08 &2.86 &4.06E-07 &2.88
&1.19E-08 &3.28 &1.08E-07 &3.32\\\hline

\end{tabular}
\end{center}
\caption{\label{tab: acc_test} Example \ref{ex:smooth}, accuracy test for smooth solution $u$ \eqref{solution:smooth} at $T =1$. }
\end{table}

\begin{table}[!htp]
\begin{center}
\begin{tabular}{|c|c|cccc|cccc|}
  \hline

          &   N       &$ \norm{v-v_h}_{L^2}$    & order   &$ \norm{v-v_h}_{L^\infty}$   & order
                     &$ \norm{v-v_h}_{L^2}$    & order   &$ \norm{v-v_h}_{L^\infty}$   & order\\\hline
           &           &  \multicolumn{4}{c|}{The energy stable DG scheme} & \multicolumn{4}{c|}{The energy stable integration DG scheme}\\\hline

 $P^1$&  20  &1.68E-03 &-- &7.80E-03 &--
             &1.56E-04 &-- &7.64E-04 &--\\
       &  40  &3.80E-04 &2.14 &1.71E-03 &2.19
              &1.97E-05 &2.99 &9.76E-05 &2.97\\
       &  80  &9.22E-05 &2.04 &3.94E-04 &2.12
              &2.46E-06 &3.00 &1.22E-05 &3.00\\
       & 160  &2.30E-05 &2.01 &9.39E-05 &2.07
              &3.08E-07 &3.00 &1.53E-06 &3.00\\
       & 320  &5.75E-06 &2.00 &2.29E-05 &2.04
              &3.85E-08 &3.00 &1.90E-07 &3.00\\\hline

$P^2$  &  20  &5.13E-05 &-- &2.62E-04 &--
              &6.57E-06 &-- &4.82E-05 &--\\

       &  40  &4.73E-06 &3.44 &2.39E-05 &3.46
              &4.41E-07 &3.90 &3.25E-06 &3.89\\
       &  80  &5.76E-07 &3.04 &2.91E-06 &3.04
              &2.62E-08 &4.07 &2.32E-07 &3.81\\
       & 160  &7.29E-08 &2.98 &3.76E-07 &2.95
              &1.60E-09 &4.03 &1.38E-08 &4.07\\
       & 320  &9.15E-09 &2.99 &4.66E-08 &3.01
              &1.09E-10 &3.88 &1.02E-09 &3.76\\\hline

\end{tabular}
\end{center}
\caption{\label{tab: acc_testv} Example \ref{ex:smooth}, accuracy test for the derivative of solution $u_x$ \eqref{solution:smooth} at $T =1$. }
\end{table}

\end{example}

\begin{example}\label{ex:shock} $\bf{ Shock \ solution}$

In this example, we consider the smooth initial data with $\gamma = -1$
\begin{equation}\label{solution: cos}
u_0(x) = -0.05\cos(2\pi x), \ x\in[0,1],
\end{equation}
which will develop a shock in finite time. To eliminate the oscillation near the shock, we follow the idea of handling the shock solutions of conservation laws  \cite{1989_Cockburn_JCP} to introduce a TVB limiter. Our numerical scheme can capture the shock without oscillation, see Figure \ref{fig:cos_shock}. Due to the similarity between the two energy stable schemes, we simply provide the approximation results of the energy stable DG scheme here.

\begin{figure}[!htb]
\begin{center}
\begin{tabular}{cc}
\includegraphics[width=0.45\textwidth]{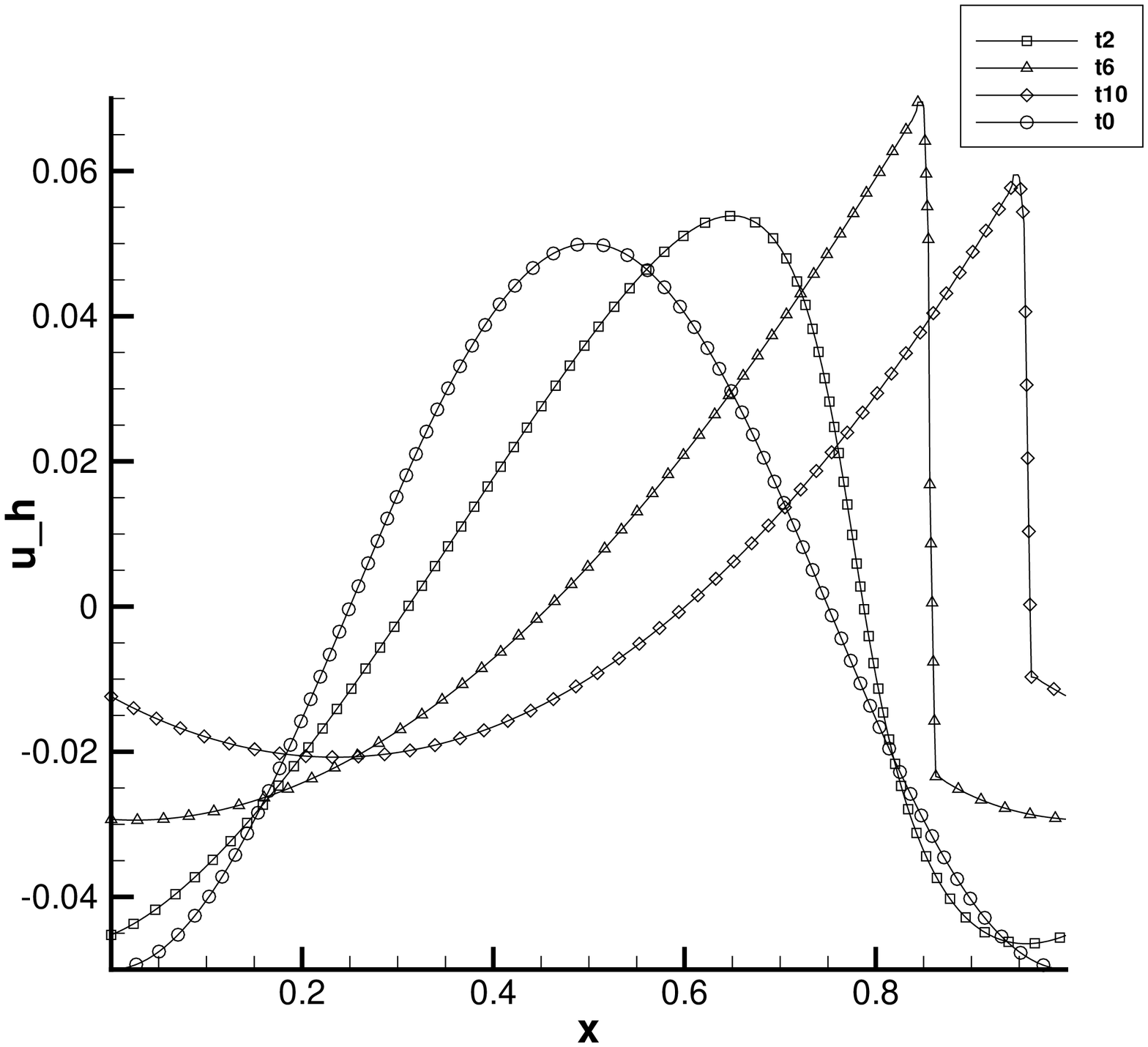}&\includegraphics[width=0.45\textwidth]{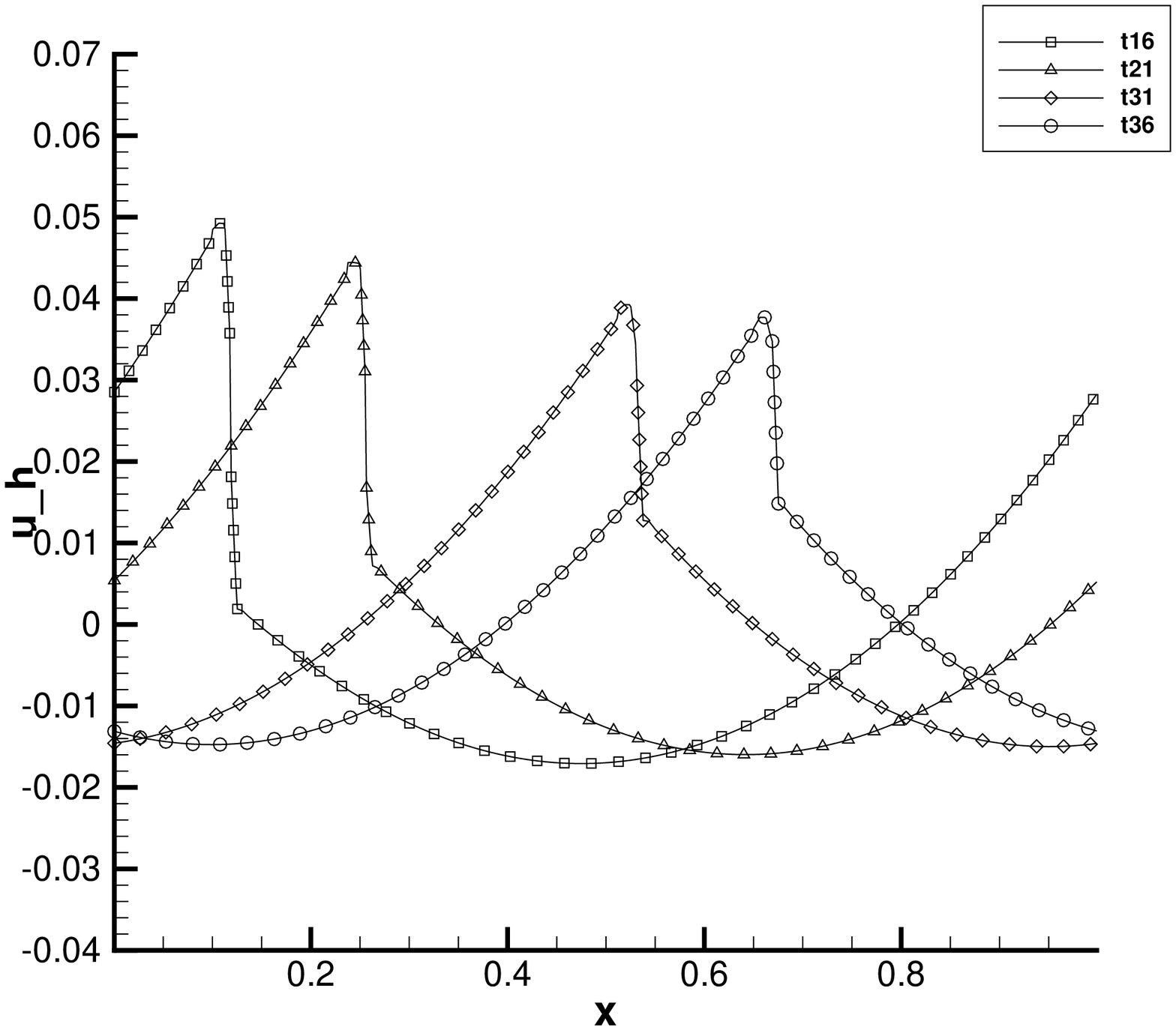}\\
\end{tabular}
\end{center}
\caption{\label{fig:cos_shock} Example \ref{ex:shock}, the process of cosine initial condition \eqref{solution: cos} at different times $T = 0$ to $T=36$ with the cells $N= 160$, $P^2$ elements. }
\end{figure}

\begin{example}\label{ex:peakon} $\bf{ Peakon \ solution}$

This example is devoted to solve a well-known traveling wave solution of the OV equation with $\gamma = -1 $. We call the corner wave whose first order derivative is finite discontinuous as a peakon solution which is the limit case of a family of smooth traveling wave solution \cite{Hunter1990_LAM, Parkes_2007_CSF, 2006_Stepanyants_CSF}. The initial data is given by
\begin{align}\label{conservative_flux}
 u_0(x) = \begin{cases}  \frac{1}{6}(x-\frac{1}{2})^2 + \frac{1}{6}(x-\frac{1}{2}) + \frac{1}{36}, \; & x\in[0,\frac{1}{2}], \\
  \frac{1}{6}(x-\frac{1}{2})^2 - \frac{1}{6}(x-\frac{1}{2}) + \frac{1}{36},\; &x\in[\frac{1}{2},1],\end{cases},
\end{align}
and the exact solution is
\begin{align}\label{solution:peakon}
u(x,t) = u_0\big(x-\frac{t}{36}\big).
\end{align}
The solution at time T = 36 will return to its initial state after a period. First, the $L^2, L^\infty$ error and convergence rate of the energy stable DG scheme are contained in Table \ref{tab: acc_test_corner}. Because of the lack of smoothness for the peakon solution, the convergence is first order for $L^2$ norm, $\frac{1}{2}$-$th$ order for $L^\infty$ norm which validates the results in \cite{2017_Coclite_BIT}. In Figure \ref{fig:corner}, two numerical solutions compare very well with the exact solution.

\begin{table}[!htp]
\begin{center}
\begin{tabular}{|c|c|cccc|}
  \hline
        &N       &$ \norm{u-u_h}_{L^2}$    & order   &$ \norm{u-u_h}_{L^\infty}$   & order \\\hline
$P^1$ &  20  &1.86E-04 &-- &1.98E-03 &-- \\

&  40  &9.96E-05 &0.90 &1.27E-03 &0.65 \\

&  80  &4.65E-05 &1.10 &8.07E-04 &0.65 \\

& 160  &2.10E-05 &1.15 &5.00E-04 &0.69\\\hline

$P^2$ &  20  &7.30E-05 &-- &6.78E-04 &-- \\

&  40  &3.22E-05 &1.18 &4.38E-04 &0.63 \\

&  80  &1.43E-05 &1.18 &3.11E-04 &0.49 \\

& 160  &6.22E-06 &1.20 &2.04E-04 &0.61 \\\hline
\end{tabular}
\end{center}
\caption{\label{tab: acc_test_corner} Example \ref{ex:peakon}, accuracy test of the energy stable DG scheme for peakon solution \eqref{solution:peakon} at $T = 36$. }
\end{table}

\begin{figure}[!htb]
\begin{center}
\begin{tabular}{c}
\includegraphics[width=0.75\textwidth]{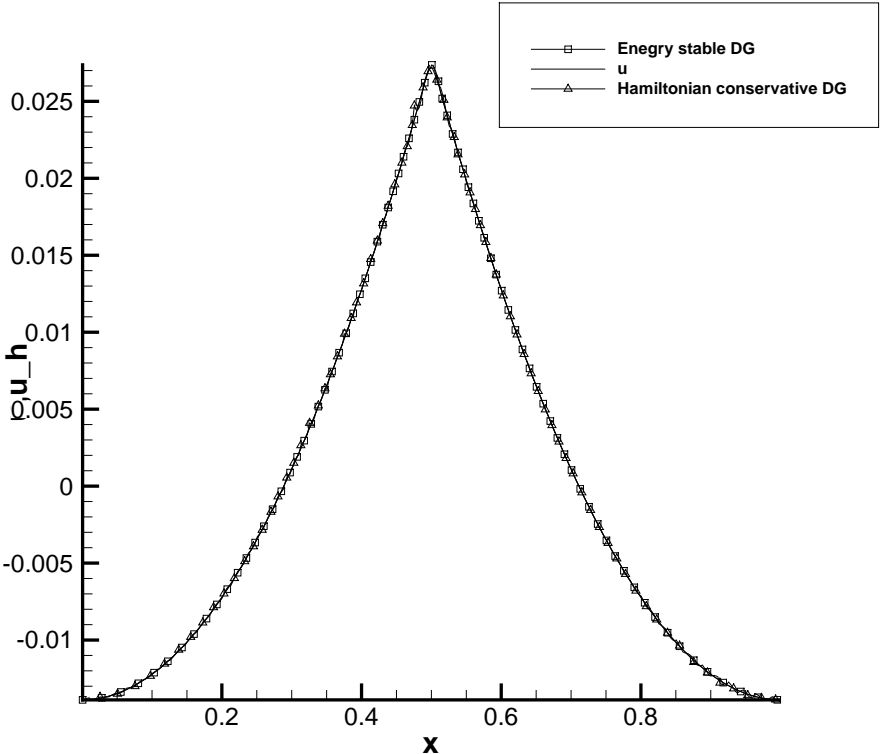}\\
\end{tabular}
\end{center}
\caption{\label{fig:corner} Example \ref{ex:peakon}, Peakon solution \eqref{solution:peakon} at T = 36 with the cells $N = 80$, $P^2$ element. }
\end{figure}

\end{example}

\end{example}

\begin{example} \label{ex:loop}{\bf Loop  and   cuspon  soliton  solutions}

This example is devoted to solve the loop and cuspon solutions for the OV equation and the two component OV system with $\gamma = -3$,
\begin{align}\label{eqn:2Ostrovsky}
\begin{cases}
(u_{t} +uu_x)_x - 3u = c(1-\rho),\\
\rho_t + (\rho u)_x = 0.
\end{cases}
\end{align}
where $c=0$, the system degenerates to the OV equation. We provide the exact solution of the OV system under the coordinate $(y,s)$,
\begin{equation}
\begin{split}
&u(y,s) = -2\ln(f)_{ss}, \\
&\rho(y,s) = (1-2(\ln f)_{ys})^{-1}, \\
&x = y - 2(\ln f)_s, \ t = s,
\end{split}
\end{equation}
which expresses the $N$ soliton solution, $f$ is the Pfaffian polynomial.

First, we use the one-soliton solution to test the error and the convergence rate,
\begin{equation}\label{reduced_soliton1}
f = 1 + e^{\eta_1}, \ \eta_1 = k_1s + \frac{3k_1}{k_1^2-c} +\eta_{10}
\end{equation}
where $k_1 = 1.0,c = 2.0$ are constants. The $L^2$, $L^{\infty}$ errors and the convergence rates of two DG methods are listed in Table \ref{tab:Ostrovsky1}, \ref{tab:Ostrovsky2}. We see that optimal error order can be both achieved for these two DG schemes \eqref{eqn:DG_reduce_O} and \eqref{eqn:DG_reduce_O_integration}.

\begin{table}[!htb]
\caption{\label{tab:Ostrovsky1} Example \ref{ex:loop}, the DG scheme:  Accuracy test for the one-soliton solution \eqref{reduced_soliton1} of the CD system \eqref{eqn:ODE_CD} at $T = 1$, the computational domain is $[-20,20]$, $k_1 = 1.0,c = 2.0$. }
\centering
\begin{tabular}{|c|c|cccc|cccc|}
  \hline
           & N       &$ \norm{u-u_h}_{L^2}$    & order   &$ \norm{u-u_h}_\infty$   & order
                     &$ \norm{\rho-\rho_h}_{L^2}$    & order   &$ \norm{\rho-\rho_h}_\infty$   & order\\\hline

$P^2$       &  20  &6.50E-03 &-- &7.13E-02 &--
 &2.04E-02 &-- &2.93E-01 &-- \\

&  40  &1.14E-03 &2.51 &2.54E-02 &1.49
 &2.32E-03 &3.14 &4.69E-02 &2.64 \\

&  80  &1.43E-04 &3.00 &2.55E-03 &3.31
 &4.82E-04 &2.27 &8.73E-03 &2.42 \\

& 160  &2.05E-05 &2.80 &4.17E-04 &2.62
 &6.38E-05 &2.92 &1.44E-03 &2.60 \\

& 320  &2.57E-06 &2.99 &5.20E-05 &3.00
 &7.87E-06 &3.02 &1.75E-04 &3.04 \\\hline

$P^3$ &  20  &2.73E-03 &-- &3.56E-02 &--
 &6.81E-03 &-- &8.69E-02 &-- \\

&  40  &2.69E-04 &3.34 &5.22E-03 &2.77
 &9.20E-04 &2.89 &1.44E-02 &2.59 \\

&  80  &2.53E-05 &3.41 &5.00E-04 &3.38
 &7.65E-05 &3.59 &1.50E-03 &3.26 \\

& 160  &1.35E-06 &4.22 &3.62E-05 &3.79
 &4.12E-06 &4.21 &1.04E-04 &3.86 \\

& 320  &8.51E-08 &3.99 &2.39E-06 &3.92
 &2.58E-07 &4.00 &7.10E-06 &3.87 \\\hline
\end{tabular}
\end{table}

\begin{table}[!htb]
\caption{\label{tab:Ostrovsky2} Example \ref{ex:loop}, the integration DG scheme: Accuracy test for the one-soliton solution \eqref{reduced_soliton1} of the CD system \eqref{eqn:ODE_CD} at $T = 1$, the computational domain is $[-20,20]$, $k_1 = 1.0,c = 2.0$. }
\centering
\begin{tabular}{|c|c|cccc|cccc|}
  \hline
           & N       &$ \norm{u-u_h}_{L^2}$    & order   &$ \norm{u-u_h}_\infty$   & order
                     &$ \norm{\rho-\rho_h}_{L^2}$    & order   &$ \norm{\rho-\rho_h}_\infty$   & order\\\hline
$P^2$   &  40  &4.59E-04 &-- &7.01E-03 &--
 &3.89E-03 &-- &5.44E-02 &-- \\

&  80  &2.31E-05 &4.31 &4.30E-04 &4.03
 &5.39E-04 &2.85 &9.38E-03 &2.54 \\

& 160  &2.59E-06 &3.16 &6.49E-05 &2.73
 &6.10E-05 &3.14 &1.21E-03 &2.95 \\

& 320  &1.65E-07 &3.97 &4.73E-06 &3.78
 &7.71E-06 &2.98 &1.58E-04 &2.94 \\

 & 640  &1.03E-08 &3.99 &3.02E-07 &3.97
 &9.66E-07 &3.00 &1.97E-05 &3.01 \\\hline

$P^3$ &  40  &7.55E-05 &-- &1.26E-03 &--
 &7.90E-04 &-- &9.93E-03 &-- \\

&  80  &6.90E-06 &3.45 &1.30E-04 &3.28
 &3.36E-05 &4.55 &5.62E-04 &4.14 \\

& 160  &1.74E-07 &5.31 &4.19E-06 &4.96
 &4.02E-06 &3.06 &9.16E-05 &2.62 \\

& 320  &5.52E-09 &4.98 &1.44E-07 &4.86
 &2.55E-07 &3.98 &7.06E-06 &3.70 \\

 & 640  &1.74E-10 &4.98 &4.50E-09 &5.00
 &1.60E-08 &4.00 &4.54E-07 &3.96 \\\hline

\end{tabular}
\end{table}

Next, we list the expression of the two-soliton solution
\begin{equation}\label{reduced_soliton2}
\begin{split}
&f = 1 + e^{\eta_1} +e^{\eta_2} + b_{12}e^{\eta_1 + \eta_2}, \\
& \eta_i = k_is + \frac{3k_i}{k_i^2-c}y + \eta_{i0},\\
&b_{12} = \frac{(k_1-k_2)^2(k_1^2-k_1k_2+k_2^2-3c)}{(k_1+k_2)^2(k_1^2+k_1k_2+k_2^2-3c)}.
\end{split}
\end{equation}
Figure \ref{fig:reduce2_O_2cuspon_u} and \ref{fig:reduce2_O_2cuspon_p} display the elastic collision between two cuspon solitons. Referring to \cite{Feng_2017_JPMT}, the shape of solution depends on the choice of parameters $k_i$. In Figure \ref{fig:reduce_O_2loop}, we provide a 2-loop solution for the OV equation. It can be seen that our numerical schemes have good resolutions for the cuspon and loop solition solutions of the OV equation or the OV system.

\begin{figure}[!htp]
%\begin{tabular}{cc}
\begin{minipage}{0.49\linewidth}
  \centerline{\includegraphics[width=1.1\textwidth]{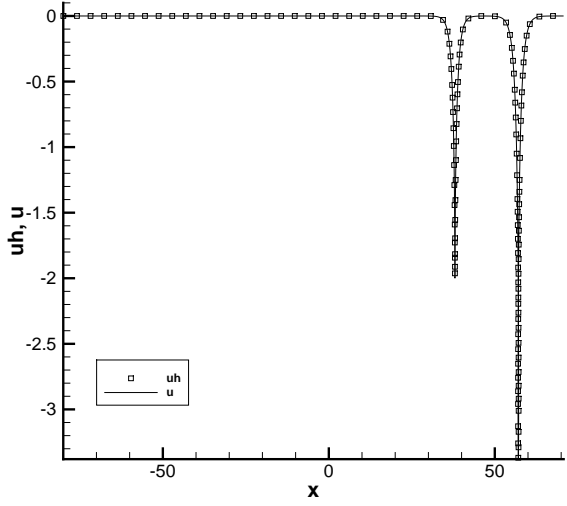}}
  \centerline{(a) t = 0.0}
\end{minipage}
\hfill
\begin{minipage}{0.49\linewidth}
  \centerline{\includegraphics[width=1.1\textwidth]{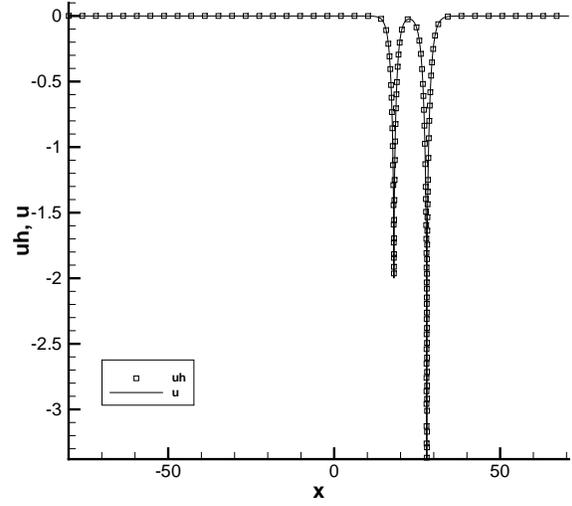}}
  \centerline{(b) t = 10.0}
\end{minipage}
\vfill
\begin{minipage}{0.49\linewidth}
  \centerline{\includegraphics[width=1.1\textwidth]{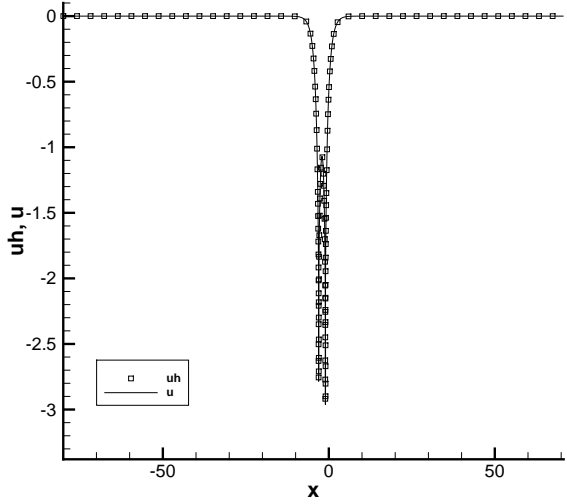}}
  \centerline{(c) t = 20.0}
\end{minipage}
\hfill
\begin{minipage}{0.49\linewidth}
  \centerline{\includegraphics[width=1.1\textwidth]{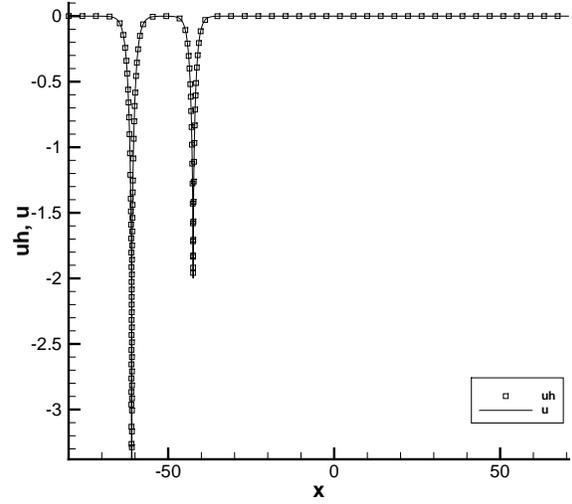}}
  \centerline{(d) t =40.0}
\end{minipage}
%\end{tabular}
\caption{ \label{fig:reduce2_O_2cuspon_u} Example \ref{ex:loop}, the two-cuspon solution $u$ of the OV system \eqref{eqn:2Ostrovsky} with the cells $N = 320$, $P^2$ elements: The parameters are $k_1 = 2.0$, $k_2$ = $2.6$, $c = -2.0, \eta_{i0} = - 20k_i. $}
\label{fig:2loop_CD}
\end{figure}

\begin{figure}[!htb]
%\begin{tabular}{cc}
\begin{minipage}{0.49\linewidth}
  \centerline{\includegraphics[width=1.1\textwidth]{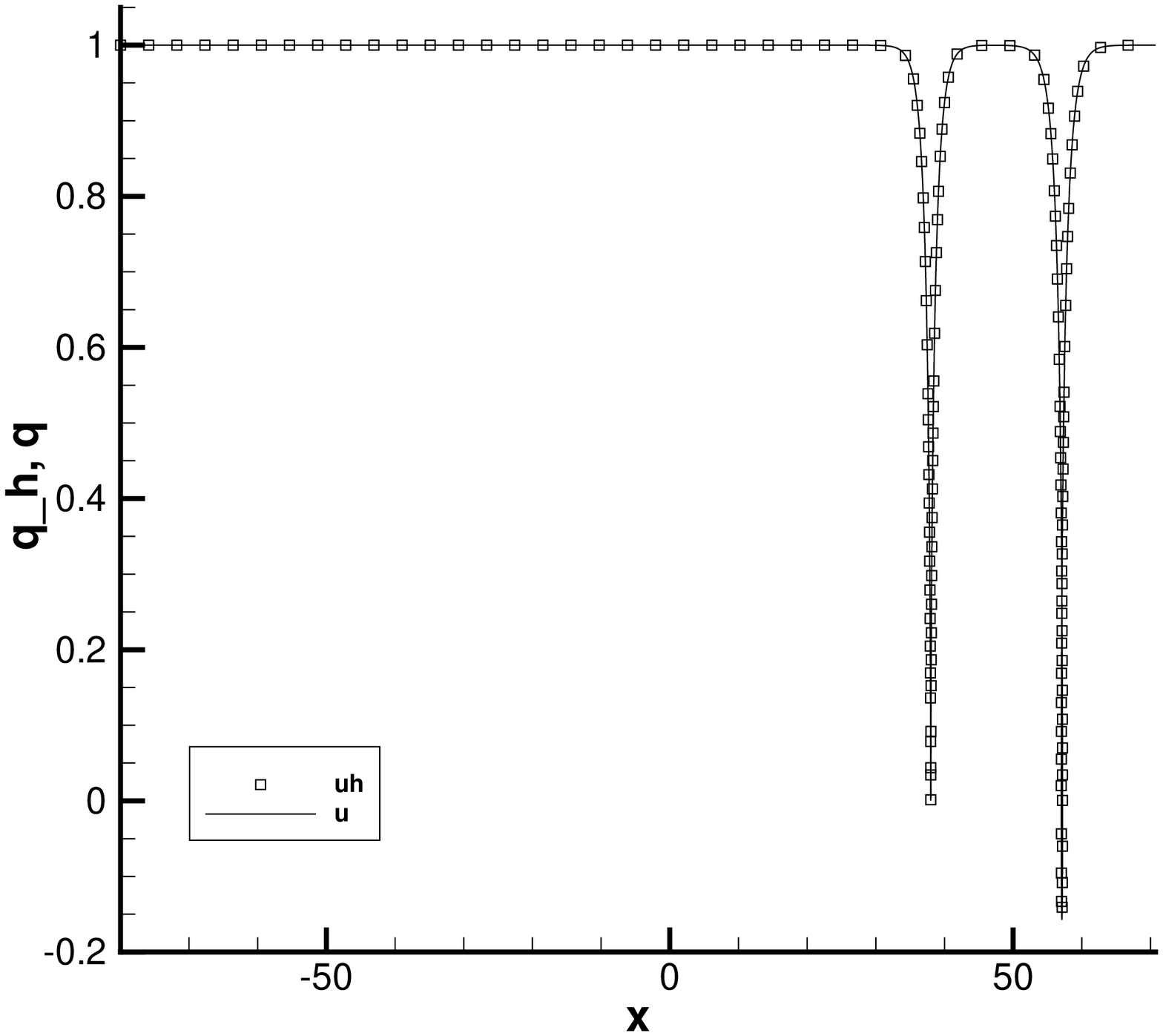}}
  \centerline{(a) t = 0.0}
\end{minipage}
\hfill
\begin{minipage}{0.49\linewidth}
  \centerline{\includegraphics[width=1.1\textwidth]{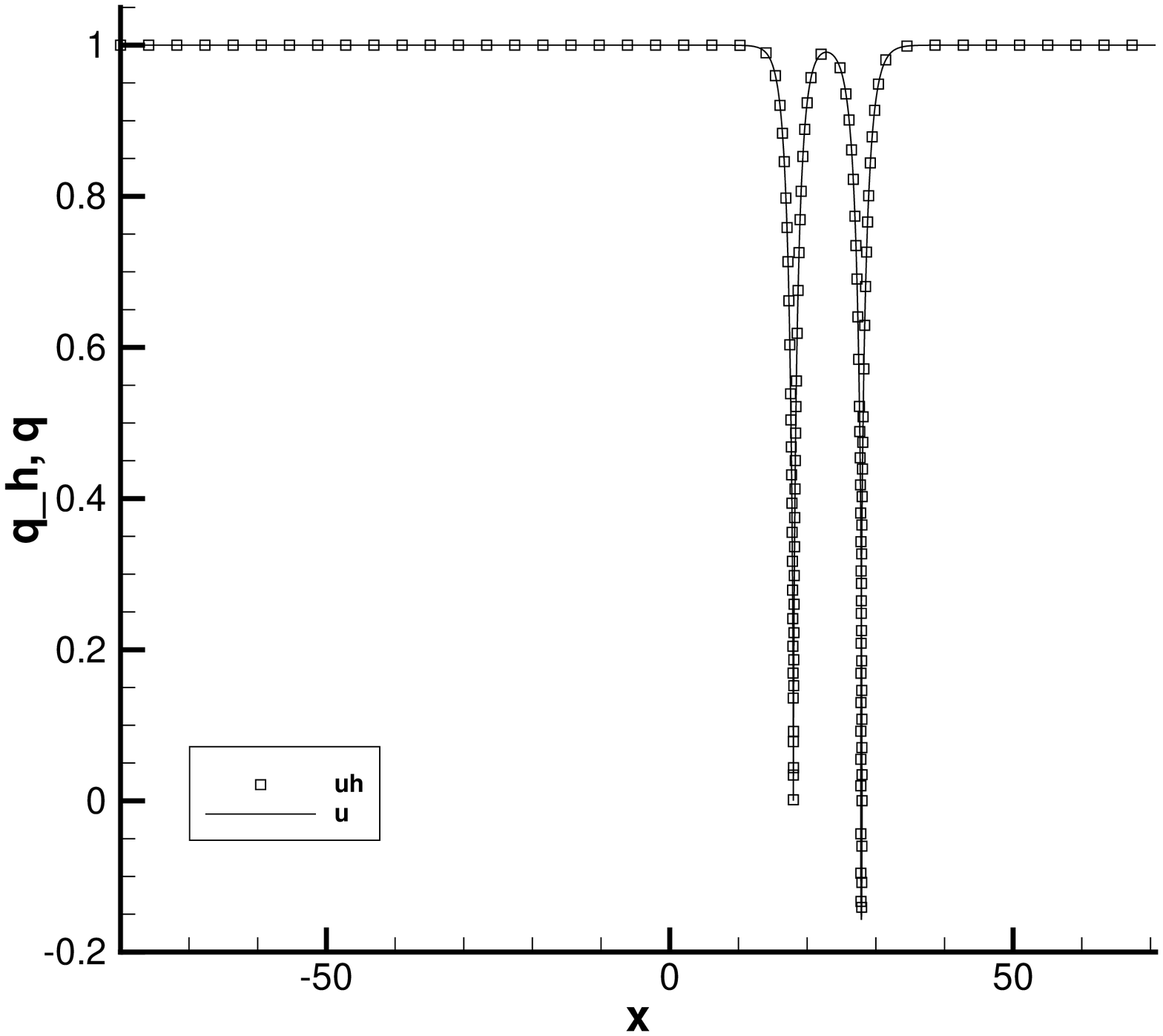}}
  \centerline{(b) t = 10.0}
\end{minipage}
\vfill
\begin{minipage}{0.49\linewidth}
  \centerline{\includegraphics[width=1.1\textwidth]{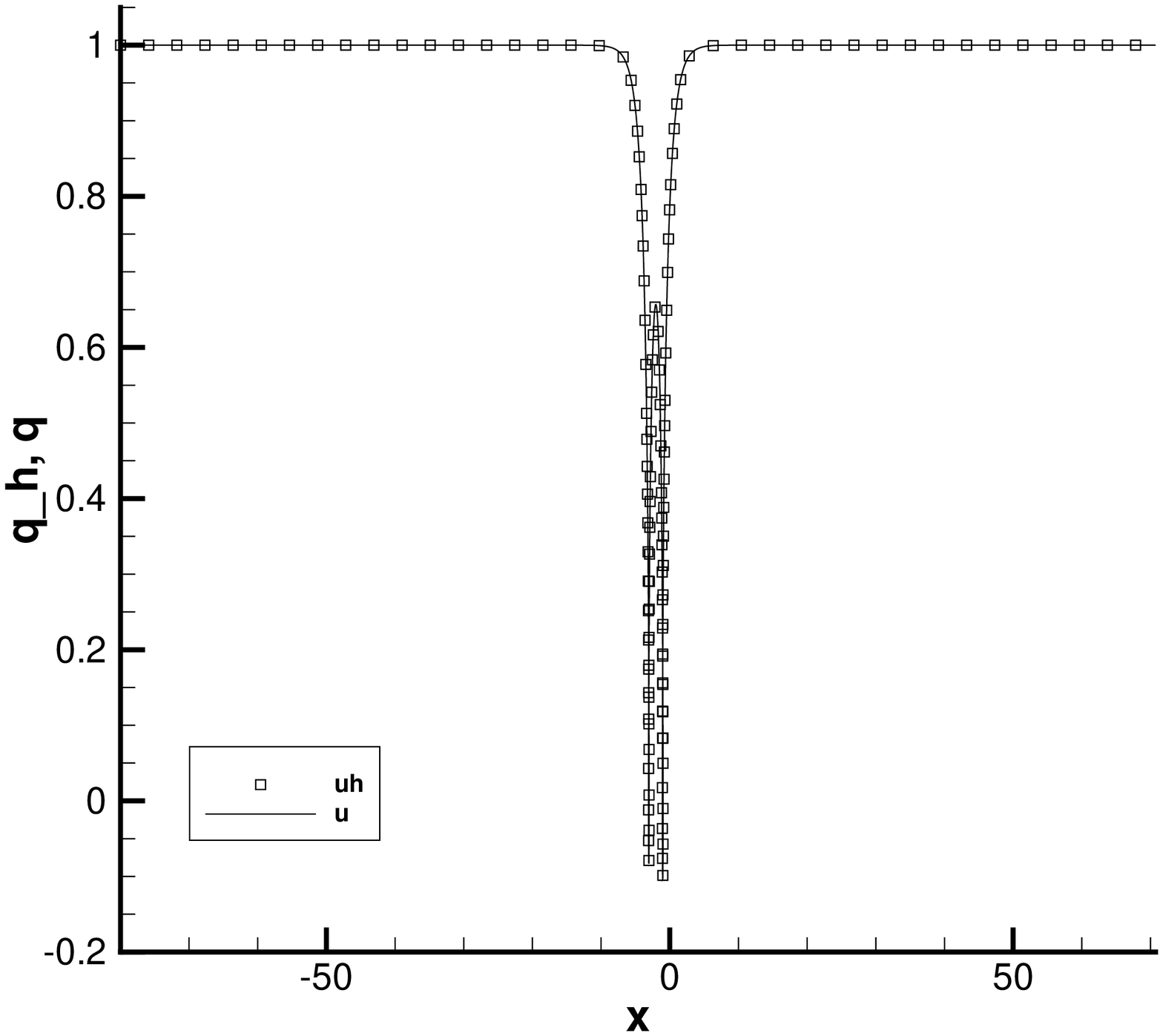}}
  \centerline{(c) t = 20.0}
\end{minipage}
\hfill
\begin{minipage}{0.49\linewidth}
  \centerline{\includegraphics[width=1.1\textwidth]{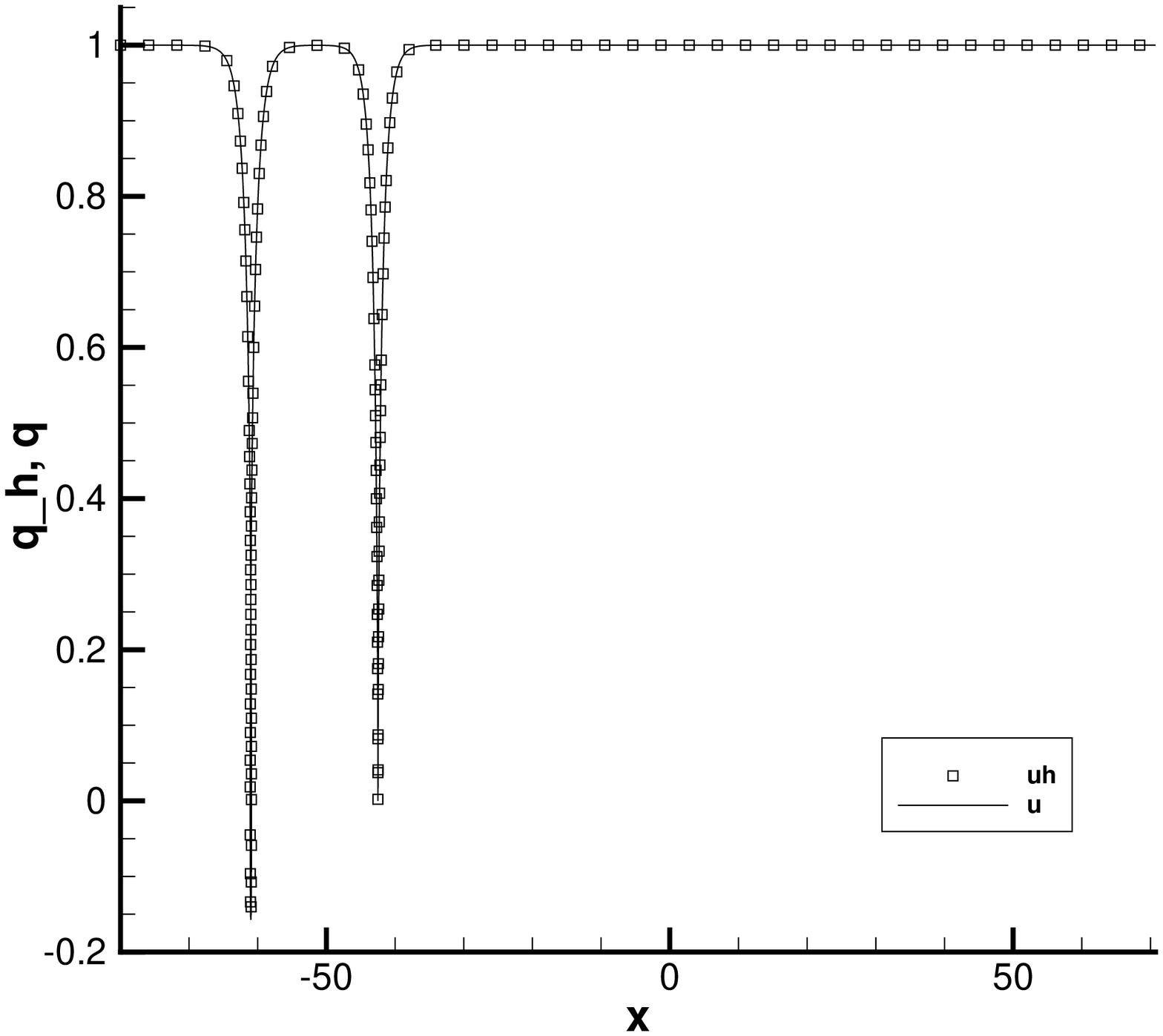}}
  \centerline{(d) t = 40.0}
\end{minipage}
%\end{tabular}
\caption{ \label{fig:reduce2_O_2cuspon_p} Example \ref{ex:loop}, the two-cuspon solution $q$ of the OV system  \eqref{eqn:2Ostrovsky} with the cells $N = 320$, $P^2$ elements: The parameters are $k_1 = 2.0$, $k_2$ = $2.6$, $c = -2.0, \eta_{i0} = - 20k_i. $}
\label{fig:2loop_CD}
\end{figure}

\begin{figure}[!htb]
%\begin{tabular}{cc}
\begin{minipage}{0.49\linewidth}
  \centerline{\includegraphics[width=1.1\textwidth]{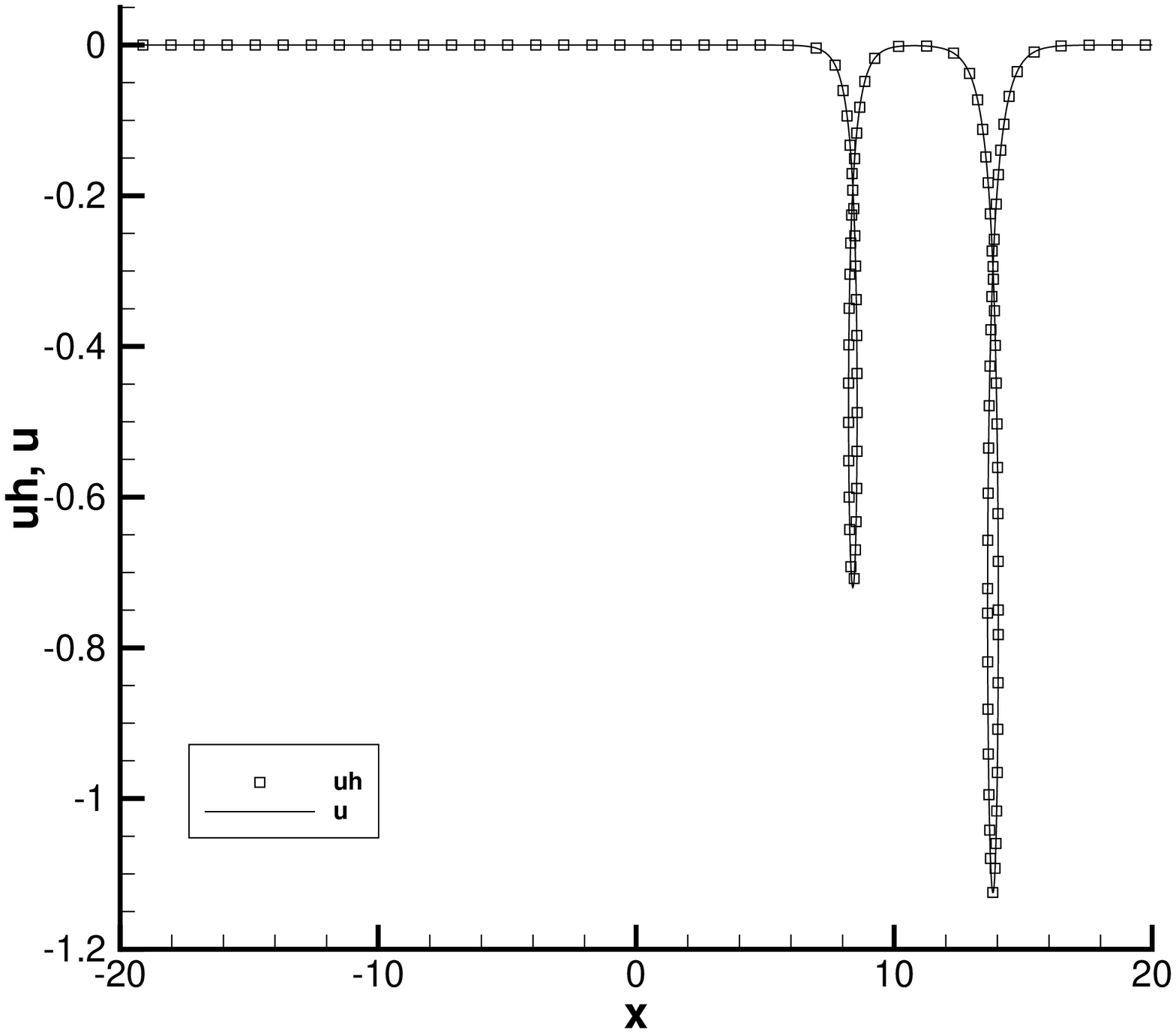}}
  \centerline{(a) t = 0.0}
\end{minipage}
\hfill
\begin{minipage}{0.49\linewidth}
  \centerline{\includegraphics[width=1.1\textwidth]{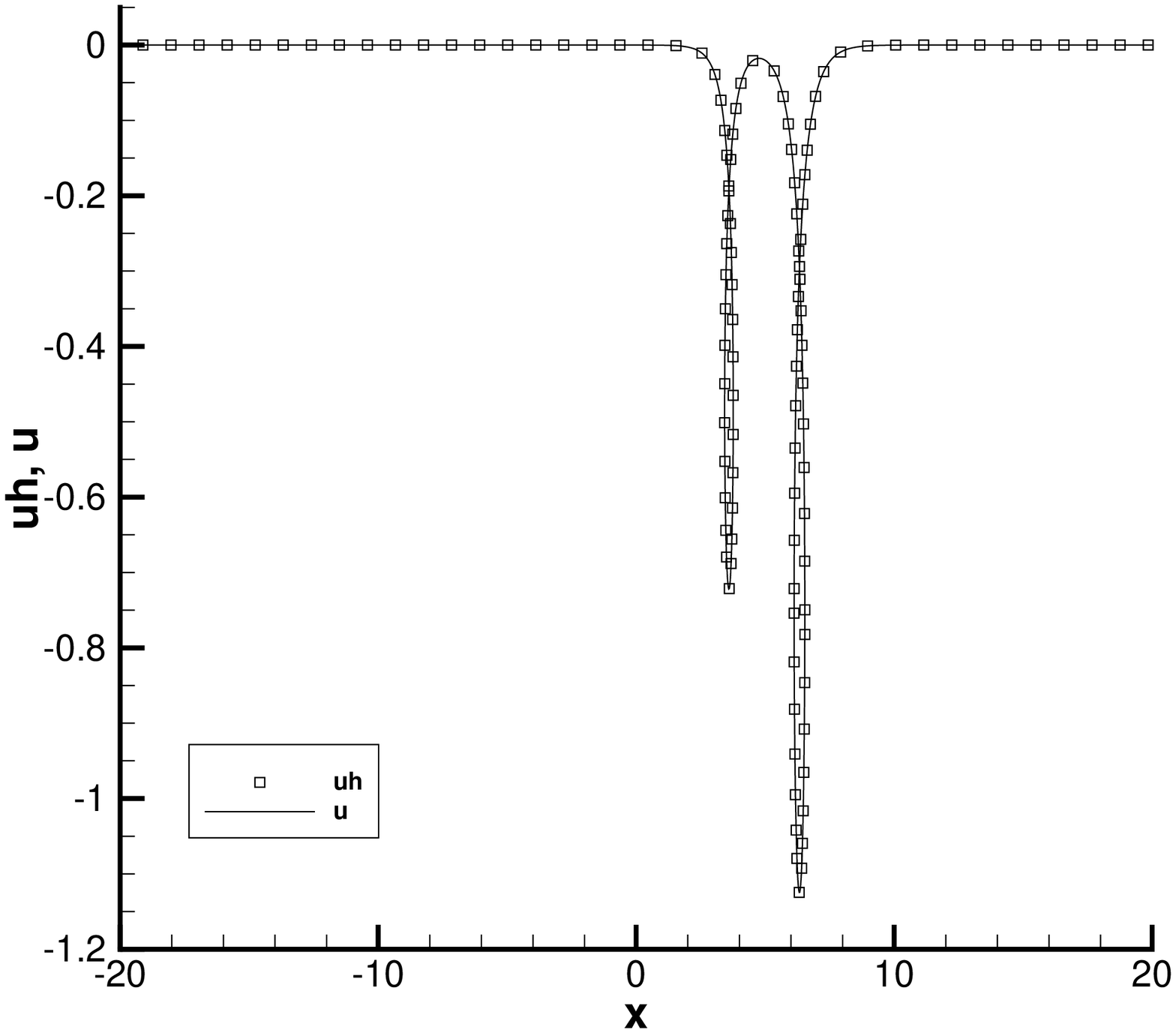}}
  \centerline{(b) t = 10.0}
\end{minipage}
\vfill
\begin{minipage}{0.49\linewidth}
  \centerline{\includegraphics[width=1.1\textwidth]{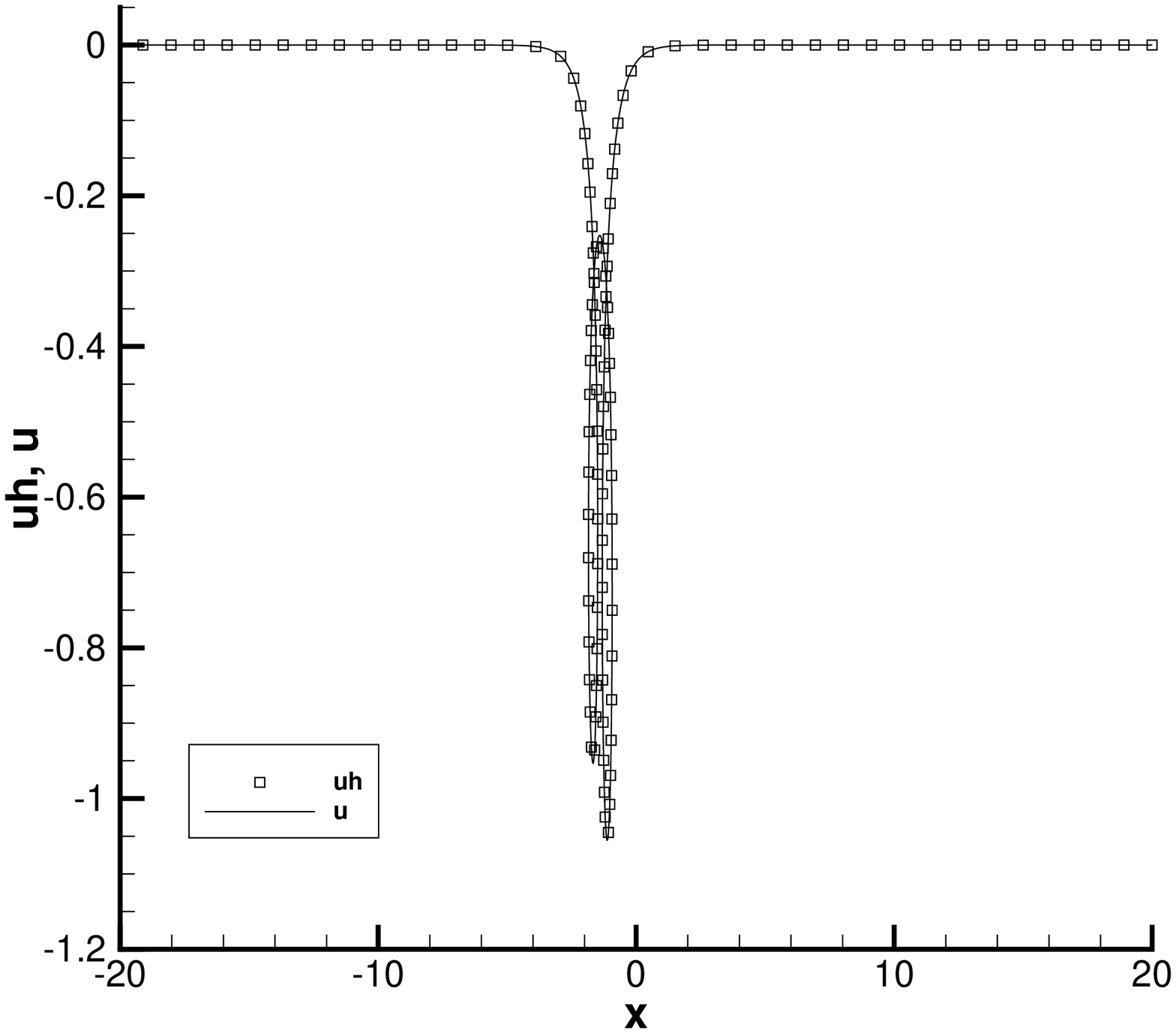}}
  \centerline{(c) t = 20.0}
\end{minipage}
\hfill
\begin{minipage}{0.49\linewidth}
  \centerline{\includegraphics[width=1.1\textwidth]{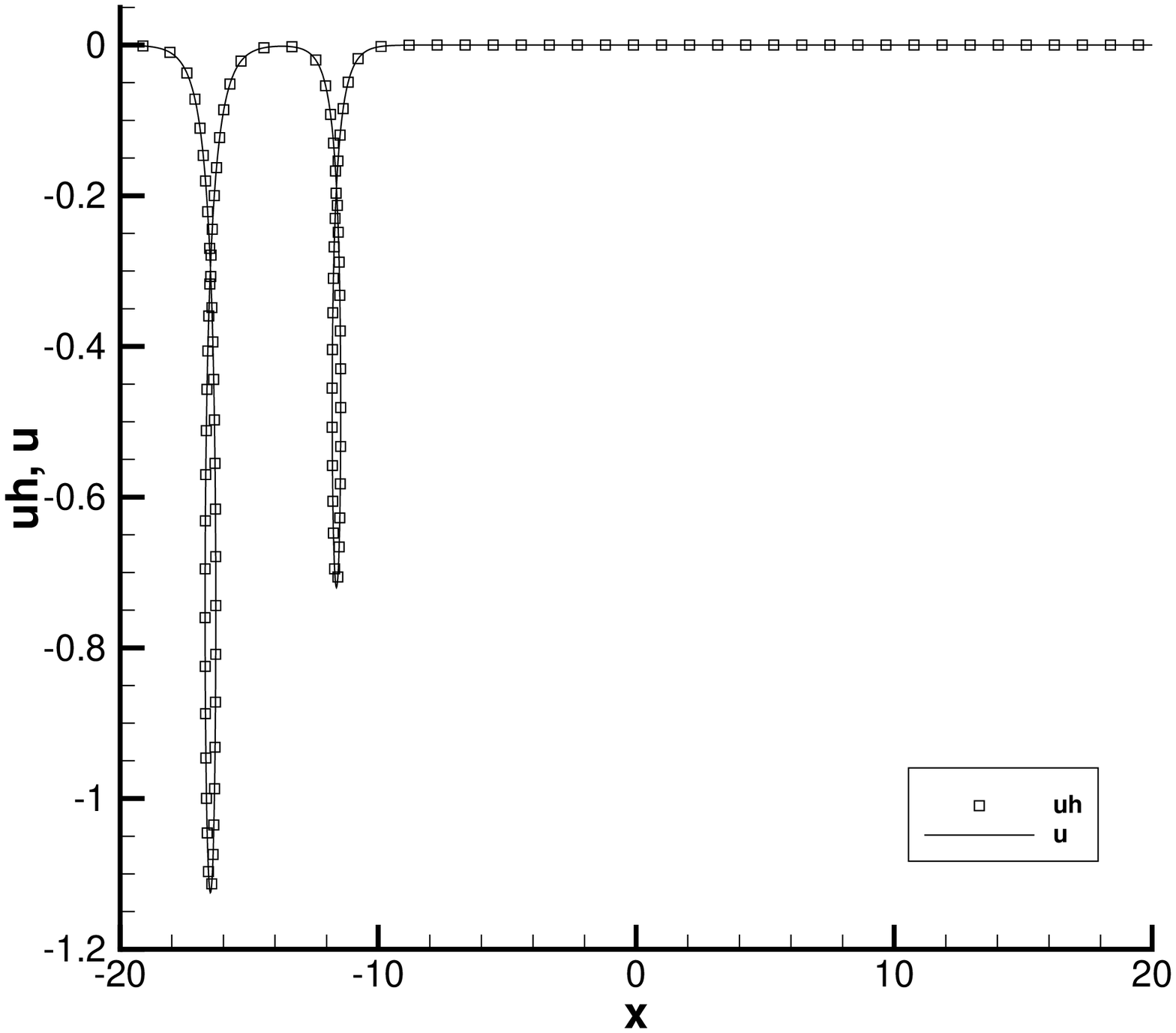}}
  \centerline{(d) t = 40.0}

\end{minipage}
%\end{tabular}
\caption{ \label{fig:reduce_O_2loop} Example \ref{ex:loop}, the two-loop solution $u$ of the OV equation with the cells $N = 320$, $P^2$ elements: The parameters are $k_1 = 1.2$, $k_2$ = $1.5$, $c = 0.0, \eta_{i0} = - 20k_i.$}
\label{fig:2loop_CD}
\end{figure}

\end{example}

\subsection{Conclusion}

In this paper, we presented  the discontinuous Galerkin methods for the OV equation. These methods can be divided into two classes: direct and indirect. Direct methods consist of the energy stable and Hamiltonian conservative DG schemes for the OV equation. The $L^2$ stability and Hamiltonian conservativeness DG schemes are proved, respectively. Based on $L^2$ stability, we also give the suboptimal error estimates of the energy stable DG scheme and the energy stable integration DG scheme. Indirect methods, composed of the DG scheme and the integration DG scheme for the CD system obtain the profile of solutions of the OV equation via the hodograph transformation. Numerical experiments are provided to demonstrate the accuracy and capability of the DG schemes, including shock solution, peakon, cuspon and loop soliton solutions, in addition to smooth solutions.


\begin{thebibliography}{99}

\bibitem{Bassi1997_JCP} F. Bassi and S. Rebay. A high-order accurate discontinuous finite element method for the numerical solution of the compressible Navier-Stokes equations. Journal of Computational Physics, 1997, 131(2): 267-279. %r10
\bibitem{Brunelli_2013_CNSNS}    J.C. Brunelli and  S. Sakovich. Hamiltonian structures for the Ostrovsky-Vakhnenko equation. Communications in Nonlinear Science and Numerical Simulation, 2013, 18(1): 56-62.
\bibitem{Bona2013_MC} J. Bona, H. Chen, O. Karakashian and Y. Xing. Conservative, discontinuous Galerkin methods for the generalized Korteweg-de Vries equation. Mathematics of Computation, 2013, 82(283): 1401-1432. %r14
\bibitem{1975_Ciarlet_NH} P.G. Ciarlet. The Finite Element Method for Elliptic Problems. North Holland, 1975.
\bibitem{Shu1998_Siam} B. Cockburn and C.-W. Shu. The local discontinuous Galerkin method for time-dependent convection-diffusion systems. SIAM Journal on Numerical Analysis, 1998, 35(6): 2440-2463.%r9
\bibitem{1989_Cockburn_JCP} B. Cockburn, S.Y. Lin and C.-W. Shu. TVB Runge-Kutta local projection discontinuous Galerkin finite element method for conservation laws III: one-dimensional systems. Journal of Computational Physics, 1989, 84(1): 90-113.       .
 \bibitem{2017_Coclite_BIT} G.M. Coclite, J. Ridder and N.H. Risebro. A convergent finite difference scheme for the Ostrovsky-Hunter equation on a bounded domain. BIT Numerical Mathematics, 2017, 57(1): 93-122.
\bibitem{Feng_2015_JPMT} B.F. Feng, K. Maruno and Y. Ohta. Integrable semi-discretizations of the reduced Ostrovsky equation. Journal of Physics A: Mathematical and Theoretical, 2015, 48(13): 135203.
\bibitem{Feng_2017_JPMT}  B.F. Feng, K. Maruno and Y. Ohta. A two-component generalization of the reduced Ostrovsky equation and its integrable semi-discrete analogue. Journal of Physics A: Mathematical and Theoretical, 2017, 50(5): 055201.
\bibitem{Gottlieb_2001_SIAM} S. Gottlieb, C.-W. Shu and E. Tadmor. Strong stability-preserving high-order time discretization methods. SIAM review, 2001, 43(1): 89-112.
\bibitem{Gui_2007_CPDE}   G. Gui and Y. Liu. On the Cauchy problem for the Ostrovsky equation with positive dispersion. Communications in Partial Differential Equations, 2007, 32(12): 1895-1916.
\bibitem{Grimshaw_2012_SAM}    R.H.J. Grimshaw, K. Helfrich and E.R. Johnson. The reduced Ostrovsky equation: integrability and breaking. Studies in Applied Mathematics, 2012, 129(4): 414-436.

\bibitem{Hunter1990_LAM}   J.K. Hunter. Numerical solutions of some nonlinear dispersive wave equations. Lect. Appl. Math, 1990, 26: 301-316.
\bibitem{Xing2016_CiCP} O. Karakashian and Y.L. Xing. A posteriori error estimates for conservative local discontinuous Galerkin methods for the generalized Korteweg-de Vries equation. Communications in Computational Physics, 2016, 20(01): 250-278.
\bibitem{Levy2004_JCP} D. Levy, C.-W. Shu and J. Yan. Local discontinuous Galerkin methods for nonlinear dispersive equations. Journal of Computational Physics, 2004, 196(2): 751-772.
\bibitem{Linares_2006_JDE} F. Linares and A. Milan\'{e}s. Local and global well-posedness for the Ostrovsky equation. Journal of Differential Equations, 2006, 222(2): 325-340.
\bibitem{Liu_2010_JMA} Y. Liu, D. Pelinovsky and A. Sakovich. Wave breaking in the Ostrovsky-Hunter equation. SIAM Journal on Mathematical Analysis, 2010, 42(5): 1967-1985.

 \bibitem{Parkes_1993_JPA}  E.J. Parkes. The stability of solutions of Vakhnenko's equation. Journal of Physics A: Mathematical and General, 1993, 26(22): 6469.

 \bibitem{Parkes_2007_CSF} E.J. Parkes. Explicit solutions of the reduced Ostrovsky equation. Chaos, Solitons and Fractals, 2007, 31(3): 602-610.
 \bibitem{Parkes_2008_CSF} E.J. Parkes. Some periodic and solitary travelling-wave solutions of the short-pulse equation. Chaos, Solitons and Fractals, 2008, 38(1): 154-159.
 \bibitem{Reed1973} W.H. Reed and T.R. Hill. Triangular mesh methods for the neutron transport equation. Los Alamos Report LA-UR-73-479, 1973.
 \bibitem{2018_Ridder_BIT}   J. Ridder and A.M. Ruf. A convergent finite difference scheme for the Ostrovsky-Hunter equation with Dirichlet boundary conditions. BIT Numerical Mathematics, 2018: 1-22.
 \bibitem{1988_Shu_JCP} C.-W. Shu and S. Osher. Efficient implementation of essentially non-oscillatory shock-capturing schemes. Journal of Computational Physics, 1988, 77(2): 439-471.
 \bibitem{2006_Stepanyants_CSF} Y.A. Stepanyants. On stationary solutions of the reduced Ostrovsky equation: Periodic waves, compactons and compound solitons. Chaos, Solitons and Fractals, 2006, 28(1): 193-204.
\bibitem{Vakhnenko_1992_JPA} V.A. Vakhnenko. Solitons in a nonlinear model medium. Journal of Physics A: Mathematical and General, 1992, 25(15): 4181.
\bibitem{Vakhnenko_1998_non}  V.O. Vakhnenko and E.J. Parkes. The two loop soliton solution of the Vakhnenko equation. Nonlinearity, 1998, 11(6): 1457.
\bibitem{Varlamov_2004_DCDs}  V. Varlamov and Y. Liu. Cauchy problem for the Ostrovsky equation. Discrete and Continuous Dynamical Systems-A, 2004, 10(3): 731-753.




\bibitem{Xu2004_JCM} Y. Xu and C.-W. Shu. Local discontinuous Galerkin methods for three classes of nonlinear wave equations. Journal of Computational Mathematics, 2004: 250-274.
\bibitem{Xu2005_JCP} Y. Xu and C.-W. Shu. Local discontinuous Galerkin methods for nonlinear Schr\"{o}dinger equations. Journal of Computational Physics, 2005, 205(1): 72-97.
\bibitem{Xu2005_PDNP} Y. Xu and C.-W. Shu. Local discontinuous Galerkin methods for two classes of two-dimensional nonlinear wave equations. Physica D: Nonlinear Phenomena, 2005, 208(1): 21-58.
\bibitem{Xu2006_CMAME} Y. Xu and C.-W. Shu. Local discontinuous Galerkin methods for the Kuramoto-Sivashinsky equations and the Ito-type coupled KdV equations. Computer Methods in Applied Mechanics and Engineering, 2006, 195(25): 3430-3447.
\bibitem{Xu2007_CMAME} Y. Xu and C.-W. Shu. Error estimates of the semi-discrete local discontinuous Galerkin method for nonlinear convection-diffusion and KdV equations. Computer Methods in Applied Mechanics and Engineering, 2007, 196(37-40): 3805-3822.
\bibitem{Xu2010_CiCP} Y. Xu and C.-W. Shu. Local discontinuous Galerkin methods for high-order time-dependent partial differential equations, Communications in Computational Physics, 2010, 7: 1-46.
\bibitem{Xia2010JCP} Y. Xia, Y. Xu and C.-W. Shu. Local discontinuous Galerkin methods for the generalized Zakharov system. Journal of Computational Physics, 2010, 229(4):1238-1259.
\bibitem{Xia2014CiCP} Y. Xia and Y. Xu. A conservative local discontinuous Galerkin method for the Schr\"{o}dinger-KdV system. Communications in Computational Physics, 2014, 15(4): 1091-1107.
\bibitem{Yan2002_Siam} J. Yan and C.-W. Shu. A local discontinuous Galerkin method for KdV type equations. SIAM Journal on Numerical Analysis, 2002, 40(2): 769-791.%r11
\bibitem{Yan2002_JSC} J. Yan and C.-W. Shu. Local discontinuous Galerkin methods for partial differential equations with higher order derivatives. Journal of Scientific Computing, 2002, 17(1-4): 27-47.
 \bibitem{Zhang2018_SIAM} Q. Zhang and C.-W. Shu. Error estimates to smooth solutions of Runge--Kutta discontinuous Galerkin methods for scalar conservation laws. SIAM Journal on Numerical Analysis, 2004, 42(2): 641-666.
 \bibitem{Zhang2019_Jsc} C. Zhang, Y. Xu and Y. Xia. Local discontinuous Galerkin methods for the $\mu$-Camassa-Holm and $\mu$-Degasperis-Procesi equations. Journal of Scientific Computing, 79(2019): 1294-1334.
 \bibitem{Zhang2019CiCP} Q. Zhang and Y. Xia. Conservative and dissipative local discontinuous Galerkin methods for Korteweg-de Vries type equations. Communications in Computional Physics, 2019, 25: 532-563.
 \bibitem{Zhang2019_arxiv} Q. Zhang and Y. Xia. Discontinuous Galerkin methods for short pulse type equations via hodograph transformations,	arXiv:1907.07842 [math.NA].














\end{thebibliography}
\end{document}